\documentclass[a4paper,twoside]{article}

\title{Deep neural networks for inverse problems \\with pseudodifferential operators:\\ 
An application to limited-angle tomography
\thanks{Corresponding author: Luca Ratti (luca.ratti@helsinki.fi). \newline
\textbf{Funding: }{TAB, ML, LR and SS were supported by the Finnish Centre of Excellence in Inverse Modelling and Imaging, 2018-2025, decision number 312339, Academy of Finland grants 284715, 312110, 310822 and Faculty of Science ATMATH project.
MG and MP acknowledge support by INdAM-GNCS (Research Projects 2018). MG is also supported by INdAM Doctoral Programme in Mathematics and/or Applications Cofunded by Marie Sklodowska-Curie Actions (INdAM-DP-COFUND-2015), grant number 713485. MP is partially funded by the ECSEL JU programme under the PRYSTINE Project Grant Agreement No. 783190.}}}

\author{Tatiana A. Bubba\thanks{Dept.~of Maths and Stats, University of Helsinki, Helsinki, FI-00014 
  (tatiana.bubba@helsinki.fi, luca.ratti@helsinki.fi, matti.lassas@helsinki.fi, samuli.siltanen@helsinki.fi).}
\and Mathilde Galinier\thanks{Dept.~of Physics, Comp.~Sci.~and Maths, University of Modena and Reggio Emilia, Modena, IT-41125
  (mathildeemmanuelle.galinier@unimore.it, marco.prato@unimore.it)}
\and Matti Lassas$^\dagger$
\and Marco Prato$^\ddagger$
\and Luca Ratti$^\dagger$
\and Samuli Siltanen$^\dagger$
}

\usepackage{lipsum}
\usepackage{amsfonts}
\usepackage{amsmath}
\usepackage{amsthm}
\usepackage{graphicx}
\usepackage{epstopdf}
\usepackage{algorithmic}
\ifpdf
  \DeclareGraphicsExtensions{.eps,.pdf,.png,.jpg}
\else
  \DeclareGraphicsExtensions{.eps}
\fi
\usepackage{amsopn}
\usepackage{amssymb}
\usepackage{hyperref} 
\usepackage{cleveref}
\usepackage{color}
\usepackage{subcaption}
\usepackage{minibox}
\usepackage{calc}  
\usepackage{enumitem} 
\usepackage{arydshln}
\usepackage{mathrsfs}
\usepackage{float}
\usepackage[margin=1.2in]{geometry}
\usepackage{fancyhdr}

\crefname{hypothesis}{Hypothesis}{Hypotheses}
\newtheorem{theorem}{Theorem}[section]

\newtheorem{lemma}[theorem]{Lemma}
\newtheorem{proposition}[theorem]{Proposition}

\newtheorem{rem}[theorem]{Remark}

\numberwithin{equation}{section}

\newcommand{\diam}{\operatorname{diam}}

\newcommand{\layer}{\Lambda}

\newcommand{\N}{\mathbb{N}}

\newcommand{\R}{\mathbb{R}}

\newcommand{\Span}{\operatorname{span}}
\newcommand{\supp}{\operatorname{supp}}

\def\hat{\widehat}
\def\tilde{\widetilde}
\def \bfo {\begin {eqnarray*} }
\def \efo {\end {eqnarray*} }
\def \ba {\begin {eqnarray*} }
\def \ea {\end {eqnarray*} }
\def \beq {\begin {eqnarray}}
\def \eeq {\end {eqnarray}}
\def \supp {\hbox{supp }}
\def \diam {\hbox{diam }}

\newcommand{\wt}[1]{w_\rho^{(#1)}}
\def\vN{{p,q}}

\DeclareMathOperator*{\argmin}{arg\,min}

\begin{document}

\maketitle

\begin{abstract}
We propose a novel convolutional neural network (CNN), called $\Psi$DONet, designed for learning pseudodifferential operators ($\Psi$DOs) in the context of linear inverse problems. Our starting point is the Iterative Soft Thresholding Algorithm (ISTA), a well-known algorithm to solve sparsity-promoting minimization problems. We show that, under rather general assumptions on the forward operator, the unfolded iterations of ISTA can be interpreted as the successive layers of a CNN, which in turn provides fairly general network architectures that, for a specific choice of the parameters involved, allow to 
reproduce ISTA, or a perturbation of ISTA for which we can bound the coefficients of the filters. 
Our case study is the limited-angle X-ray transform and its application to limited-angle computed tomography (LA-CT). In particular, we prove that, in the case of LA-CT, the operations of upscaling, downscaling and convolution, which characterize our $\Psi$DONet and most deep learning schemes, can be exactly determined by combining the convolutional nature of the limited angle X-ray transform and basic properties defining an orthogonal wavelet system. We test two different implementations of $\Psi$DONet on simulated data from limited angle geometry, generated from the ellipse data set. Both  implementations  provide  equally  good  and  noteworthy preliminary results, showing the potential of the approach we propose and paving the way to applying the same idea to other convolutional operators which are $\Psi$DOs or Fourier integral operators. \par \medskip
\end{abstract}

\section{Introduction}
\label{sec:introduction}
In the context of microlocal analysis, the theory of pseudodifferential operators ($\Psi$DOs), introduced by Kohn and Nirenberg in 1965, and Fourier integral operators (FIOs), defined by H\"{o}rmander in 1971, finds remarkable applications in many fields of Mathematics, from spectral theory to general relativity, from the study of the behavior of chaotic systems to scattering  theory and inverse problems~\cite{Hormander85}.

A prominent example in the inverse problem field is given by the X-ray transform or, in the two-dimensional case, Radon transform: 
\begin{equation}
R(u)(s,\omega) \, = \, \int_{-\infty}^{\infty} u(s\omega^{\perp} + t\omega) \, \text{d}t 
\qquad s \in \R, \; \omega, \omega^{\perp} \in S^1
\label{eq:RadonTr}
\end{equation}
where $\omega^{\perp}$ denotes the vector in the unit sphere $S^1$ obtained by rotating $\omega$ counterclockwise by $90^{\circ}$~\cite{natterer2001}. It is possible to show (see, \textit{e.g.}, \cite{Quinto06}) that the normal operator $R^*R$ of the Radon transform $R$ is an elliptic $\Psi$DO of order $-1$ and a convolutional operator associated with the Calder\'{o}n-Zygmund kernel $K(x,y) \, = \, \frac{1}{|x-y|}$ for $x\not=y$.
When the direction vector $\omega$ is restricted within a limited angular range $[-\Gamma, \Gamma]$, the normal operator $R_{\Gamma}^* R_{\Gamma}$ of the limited angle Radon transform $R_{\Gamma}$ is a convolutional operator associated with the kernel
\[
K(x,y) \, = \, \frac{1}{|x-y|} \, \chi_{\Gamma}(x-y)
\qquad \text{for } \; x\not=y,
\]
where $\chi_{\Gamma}$ denotes the indicator function of the cone in $\R^2$ between the angles $-\Gamma$ and $\Gamma$. The operator $R_{\Gamma}^* R_{\Gamma}$ is no longer a $\Psi$DO, but it  belongs to the wider class of FIOs, which includes operators associated with a kernel showing some discontinuities along lines~\cite{Hormander85}.

The inverse problem arising from the limited angle Radon transform, \textit{i.e.}, limited-angle computed tomography (LA-CT), appears frequently in practical applications, such as dental tomography~\cite{Kolehmainen03}, damage detection in concrete structures~\cite{heiskanen1991}, breast tomosynthesis~\cite{Zhang06} or electron tomography~\cite{Fanelli08}. In this framework, microlocal analysis is used to predict which singularities, that is, sharp features of the object being imaged, can be reconstructed in a stable way from limited angle measurements~\cite{Borg2018,Frikel13a,Krishnan2015,Quinto93}. In practice, thanks to microlocal analysis, we are able to read the part of the wavefront set of the target corresponding to the missing angular range from the measurement geometry. 

Even with this fundamental information, the task of robustly recovering the unknown quantity of interest from such partial indirect measurement is a challenging one, due to the ill-posedness 
of the CT problem, which is even more severe because of the limited angular range~\cite{Davison83}. 
As a result, classical methods, such as the filtered backprojection (FBP)~\cite{natterer2001}, yield poor performances. Traditional inversion methods of the form~\eqref{eq:minu}-\eqref{eq:minN}, based on complementing the insufficient measurements by imposing \textit{a priori} information on the solution, define effective regularization methods which generally allow for accurate reconstructions from fewer tomographic measurements than usually required by standard methods like FBP. 
In more recent years, machine learning approaches, in particular, deep learning, with convolutional neural networks (CNNs) being the most prominent design in the context of imaging, are increasingly impacting the field of inverse problems~\cite{Arridge19}, and (LA-)CT is no exception (see, in particular, \cite[section 4]{Arridge19} for an overview of learning approaches from a functional analytic regularization perspective and \cite[section 7]{Arridge19} for their applicability to prototypical examples of inverse problems, including CT). 
The majority of recent data-driven approaches for LA-CT focuses on recovering or inpainting the missing part of the wavefront set from the measured data (see, \textit{e.g.}, \cite{Bubba19,Tovey19} and the references therein for a thorough review of model-based and data-driven approaches based on sparsifying transforms and edge-preserving regularizers in the context of LA-CT).

In this paper, we are \textit{not} interested in designing an(other) approach for inferring the missing wedge in LA-CT, but rather we aim at investigating neural networks inspired by FIOs and $\Psi$DOs, for which LA-CT is a case study. 
Our starting point is the traditional sparsity-based minimization problem of the form~\eqref{eq:minu}-\eqref{eq:minN}. A well-known technique for its solution is the Iterative Soft Thresholding Algorithm (ISTA), introduced in 2004 in the seminal paper by Daubechies, Defrise and De Mol~\cite{D}. The convergence result in the paper relies on the assumption that the sparsifying system forms an orthogonal basis, as it is the case for many families of wavelets~\cite{mallat1999wavelet}. ISTA iteratively creates the sequence $\{w^{(n)}\}_{n=1}^N$ as follows:
\begin{equation}
 \label{eq:ISTAitintro}   
    w^{(n)} \, = \, \mathcal{S}_{\lambda/L}\left(w^{(n-1)} - \frac{1}{L} K^{(n)} w^{(n-1)} + \frac{1}{L} b^{(n)} \right),
\end{equation}
where, in our case, $K^{(n)}=W R_{\Gamma}^* R_{\Gamma} W^*$, with $W$ wavelet transform associated with an orthogonal family, $b^{(n)}=W R_{\Gamma}^* m$ and $S_\beta(w)$ is the (component-wise) soft-thresholding operator (see equations~\eqref{eq:ISTA} and~\eqref{eq:RESnet} for all the details). It is well-known that the unrolled iterations of ISTA can be considered as the layers of a neural network. Learned ISTA (LISTA), introduced in~\cite{Gregor}, and ISTA-Net, introduced in~\cite{Zhang18}, are examples of neural networks obtained by laying out the operations of ISTA for a few iterations. The major difference with our approach is that LISTA and ISTA-Net are not CNNs. In~\cite{Jin17} the authors investigate the relationship between CNNs and iterative optimization methods, including ISTA, for the case of normal operators associated with a forward model which is a convolution. However, the resulting U-net, FBPConvNet, does not aim at imitating an unrolled version of an iterative method, which makes it fundamentally different in spirit to the methodology we propose.
Indeed, the goal of our work is to show that, under some assumptions on the operator $R_{\Gamma}W^*$, it is possible to interpret the operations in~\eqref{eq:ISTAitintro} as a layer of a CNN, which in turn provides fairly general network architectures that allow to recover standard ISTA for a specific choice of the parameters involved.

Motivated by this, we propose a new CNN, which we name \textbf{$\Psi$DONet}, aimed at learning convolutional FIOs and $\Psi$DOs. 
The key feature of $\Psi$DONet is that we split the convolutional kernel into $K=K_0+K_1$ where $K_0$ is the known part of the model (in the limited angle case, $K_0=R_{\Gamma}^* R_{\Gamma}$) and $K_1$ is an unknown $\Psi$DO to be determined or, better, to be learned. Basically, in $K_1$ lays the potential to add information in the reconstruction process with respect to the known part of the model $K_0$. $\Psi$DONet takes advantage of the possibility to use small filters encoding a combination of upscaling, downscaling and convolution operations, as it is common practice in deep learning. Remarkably, we prove that such operations can be exactly determined combining the convolutional nature of the limited angle Radon transform and basic properties defining an orthogonal wavelet system. While this might seem contrary to the machine learning philosophy which finds its strength in avoiding any predefined structure for neural networks, our recipe gives insight into understanding and interpreting the results of the proposed CNN, combining results from FIOs, $\Psi$DOs and classical variational regularization theory.
At the same time, the possibility to deploy such operations allows for a significant reduction of the parameters involved, especially when compared to the standard interpretation of ISTA as a recurrent neural network: this is fundamental when it comes to a practical numerical implementation of the proposed CNN. Overall, $\Psi$DONet is able to reproduce ISTA, or a perturbation of ISTA for which we can bound the coefficients of the filters, and has the potential to learn $\Psi$DO-like structures which are intrinsic to the problem at hand. 

As a proof of concept, we test $\Psi$DONet on simulated data from limited angle geometry, generated from the ellipse data set. We provide two different implementations of $\Psi$DONet: Filter-Based $\Psi$DONet ($\Psi$DONet-F), where the backprojection opererator is approximated by its filter-equivalent, and Operator-Based $\Psi$DONet ($\Psi$DONet-O), where the backprojection opererator encoded in $K_0$ is not approximated but explicitly computed. Both implementations provide equally good and noteworthy preliminary results, the main difference being  a greater computational efficiency for $\Psi$DONet-O. The improvement provided by our results, compared to standard ISTA (and classical FBP), bodes well for further numerical testing which we leave to future work. 

Finally, we stress that the contribution of our paper is mainly theoretical and is in line with current research in data-driven inversion, which combines knowledge from traditional inverse problems theory with data-driven techniques.
While in our paper we derived the result contingently to the case of limited angle Radon transform, our approach is actually very general and can be extended to any convolutional operator which is a FIO or $\Psi$DO. This is the case, for instance, of the geodesic X-ray transform~\cite{Uhlmann16}, and its applications in seismic imaging, or synthetic-aperture radar (SAR)~\cite{Oliver89}.
Finally, our paper paves the way to theoretical generalization results, in light of recent contributions like~\cite{de2019deep}.

The remainder of this paper is organized as follows: \cref{sec:theo_background} is devoted to reviewing the theoretical background of sparsity promoting regularization and the wavelet transform. In \cref{sec:unrolled_ista_theo}, we detail the key idea of our approach, namely we give a convolutional interpretation of ISTA using the wavelet transform. The neural network architecture we propose, $\Psi$DONet, is introduced in \cref{sec:algo_theo}, where we also prove our main theoretical result. Two different implementations of $\Psi$DONet, which we call Filter-Based $\Psi$DONet and Operator-Based $\Psi$DONet, are described in \cref{sec:in_practice}. Finally, we demonstrate the performance of our network by a series of numerical experiments (see \cref{sec:experiments}). Concluding remarks and future prospects are briefly summarized in \cref{sec:conclusions}. The appendices collect proofs of some of the results presented in \cref{sec:theo_background}.

\section{Theoretical background}
\label{sec:theo_background}
In this section, we collect some theoretical results which are preliminary to the main discussion of the paper. 

\subsection{Sparsity-promoting regularization via ISTA}
\label{subsec:ista_standard}

Consider the inverse problem of determining $u^\dag \in X$ from the measurements $m = A u^\dag + \epsilon$, being $A: X \rightarrow Y$ a linear bounded operator between the Hilbert spaces $X$ and $Y$. The perturbation $\epsilon \in Y$ is such that $\|\epsilon\|_Y \leq \delta$. \\
The main application we have in mind is the limited-angle Radon transform $R_\Gamma$, which is a continuous linear operator, \textit{e.g.}, from $X = L^2(\Omega)$ (being $\Omega \subset \R^2$) to $Y = L^2([-\Gamma,\Gamma]\times[-S,S])$ (see \cite[Theorem 2.10]{natterer2001}). \\
Introduce an orthonormal basis $\{ \psi_I \}_{I \in \N}$ in $X$. For later purposes, we will assume that such basis is a wavelet system.
Define $W: X \rightarrow \ell^2(\N)$ the operator associating to any $u \in X$ the sequence of its component with respect to the wavelet basis: $(Wu)_I = (u,\psi_I)_X$, where $(\cdot,\cdot)_X$ denotes the inner product in $X$.
We assume to know \textit{a priori} that the exact solution $u^\dag$ is sparse with respect to the wavelet basis ${\psi_I}$ :
\begin{equation}
    Wu^\dag = w^\dag \in \ell^0(\N).
    \label{eq:sparse}
\end{equation}
\par
The reconstruction of $u^\dag$ (or, equivalently, $w^\dag$) from the noisy measurements $m$ is in general an ill-posed problem, hence we introduce the following regularized problem:
\begin{equation}
    \label{eq:min}
\min_{w\in \ell^1(\N)} \| A W^* w - m\|_{Y}^2+\lambda \|w\|_{\ell^1},
\end{equation}
being $\lambda > 0$. The requirement $w \in \ell^1(\N)$ is in general not satisfied by any $w = Wu$, $u \in X$; hence, we define $Z \subset X$, $Z = \{u \in X: Wu \in \ell^1(\N)\}$. In particular, in the tomography application, it is possible to show that the $\ell^1$ norm of the components of the wavelet representation of a $L^2(\Omega)$ function is equivalent to the Besov norm $B^1_{1,1}(\Omega)$ (see, \textit{e.g.}, \cite[formula (A3)]{D}).
Hence, the minimization problem \eqref{eq:min} is equivalent to 
\begin{equation}
    \label{eq:minu}
\min_{u \in Z} \|Au - m\|_{Y}^2+\lambda \|u\|_{Z}.
\end{equation}
It is well known that the regularization term involving the $\ell^1$ norm is a good choice to encode the \textit{a priori} information regarding the sparsity of $w^\dag$. In particular if the noise level tends to $0$, there exists a suitable choice of $\lambda = \lambda(\delta)$ ensuring the convergence of $w_\lambda^\delta$ to $w^\dag$, being $w_\lambda^\delta$ the solution of \eqref{eq:min}. We report a result from \cite{flemming2018injectivity} which also shows that such convergence occurs with linear rate. In particular, \cite[Corollary 2]{flemming2018injectivity} does not require $w^\dag$ to satisfy a classical source condition, but relies on the sparsity assumption \eqref{eq:sparse} and on the injectivity of the operator $A$. Such property can be restrictive in some applications, and as a consequence many alternative results involve some weaker assumptions (as the well known Restricted Isometry Property); nevertheless, in our tomographic application, we can rely on the injectivity of the Radon transform, even in the limited angle case.
\begin{proposition}
Let $w^\dag$ satisfy \eqref{eq:sparse}, and suppose $A:X \rightarrow Y$ is injective. Define $w_\delta^\lambda$ a solution of problem \eqref{eq:min} associated with a regularization parameter $\lambda$ and a noise level $\delta$. For sufficiently small $\delta$, provided that $\lambda$ is chosen such that $\lambda = c_0 \delta$, then there exists a positive constant $c_1 = c_1(c_0 ,A, \| w^{\dag} \|_{\ell^0})$ such that
\begin{equation}
\| w^\dag - w_\delta^\lambda\|_{\ell^1} \leq c_1 \delta.
\label{eq:conv1}
\end{equation} 
\label{prop:conv1}
\end{proposition}
This proposition is an immediate consequence of \cite[Corollary 2]{flemming2018injectivity}, relying on \cite[Lemma 2]{flemming2018injectivity} to ensure that  $A$ is weak*-to-weak continuous.
From now on, we suppose that $\lambda$ is chosen as a linear function of $\delta$ and denote $u_\delta^\lambda$ as $u_\delta$ and $w_\delta^\lambda$ as $w_\delta$. 
\par
We now introduce a finite-dimensional approximation of the regularized problem \eqref{eq:min}. Consider the subspace $X_p \subset X$, $X_p = \Span\{\psi_I\}_{I=1}^p$, mapped by $W$ into the space $W_p = \{w \in \R^\N: \ w_I = 0 \ \forall I > p \}$ (which is isomorphic to $\R^p$). Denote by $\mathbb{P}_p$ the orthogonal projection of $\ell^2(\N)$ onto $W_p$ and by $\widetilde{\mathbb{P}}_p = W^* \mathbb{P}_p W$ the orthogonal projection of $X$ onto $X_p$. 
Moreover, we introduce an orthogonal basis $\{\varphi_j\}_{j=1}^\infty$ on $Y$ and define $Y_q = \Span\{\varphi_j\}_{j=1}^q$ and the projection $\mathbb{P}_q: Y \rightarrow Y_q$. 
For any choice of $p,q>0$, let $A_\vN$ be the representation of the operator $A$ in the subspaces $X_p, Y_q$, namely $A_\vN = \mathbb{P}_q A \widetilde{\mathbb{P}}_p^*$. Consider the following minimization problem:
\beq
\label{eq:minN}
\min_{w\in W_p} \| A_\vN W^* w- \mathbb{P}_q m\|_{Y}^2+\lambda \|w\|_{\ell^1}.
\eeq
Denote by $w_{\delta,\vN}$ a solution of \eqref{eq:minN}. We can prove the following convergence result:
\begin{proposition}
Let $w^\dag$ satisfy \eqref{eq:sparse} and $A$ be an injective operator. Suppose moreover that for a suitable choice of $p,q$ it is possible to ensure that $\| w^\dag - \mathbb{P}_p w^\dag \|_{\ell^2} \leq c_p \delta$ and $\| (I - \mathbb{P}_q)A \|_{X \rightarrow Y} \leq c_q \delta$. Then, provided that $\lambda$ is chosen as $\lambda = c_0 \delta$, there exists a positive constant $c_2$ (depending on $\|A \|, \| w^\dag\|_{\ell^1}$, on the choice of $\{\psi_I\},\{\varphi_j\}$ and on the constants $c_0,c_1$,$c_p$,$c_q$) such that:
\begin{equation}
\|  w_{\delta,\vN}  - w^\dag \|_{\ell^1} \leq c_2 \delta.
\label{eq:conv2}
\end{equation}
\label{prop:conv2}
\end{proposition}
The proof, which follows by an application of the variational source condition reported in \cite[Section 3]{flemming2018injectivity}, is reported in \cref{Appendix1}.

\begin{rem}
 Upper bounds of the kind $\| w^\dag - \mathbb{P}_p w^\dag \|_{\ell^2} \leq f(p)$ can be explicitly computed under some particular assumptions on $w^\dag$. If, for example, we suppose that $u^\dag$ is a cartoon-like image (\textit{i.e}., $u^\dag$ is a $C^2$-smooth functions apart from jump discontinuity along a finite set of $C^2-$curves) and $\{\psi_I\}$ is the Haar wavelets basis, it is well known that $\| w^\dag - \mathbb{P}_p w^\dag \|_{\ell^2} \leq p^{-1}$ (see, \textit{e.g.}, \cite[Chapter 9]{mallat1999wavelet}).  \\
On the other hand, an estimate for the term $\| (I - \mathbb{P}_q)A \|_{X \rightarrow Y}$ can be obtained by standard results of finite-rank approximation of compact operators.
For example, suppose that the operator $A$ is a compact operator. Define $\{s_j\}$ its singular values (i.e., let $\{(s_j,e_j)\}$ be the eigenvalues and eigenfunctions of $(A^*A)^{\frac{1}{2}}$) and suppose the sequence $s_j$ is non-increasingly converging to $0$. A sufficient condition for this is that $A$ is a Schatten operator of any class $p$. If we select the basis $\{\varphi_j \}$ such that $\varphi_j = U e_j$, where $U$ is the partial isometry in the polar decomposition $A = U(A^*A)^{\frac{1}{2}}$, then it holds $\| (I - \mathbb{P}_q )A \|_{X\rightarrow Y} \leq s_{q+1}$. In the case of the Radon transform in 2D, according to \cite[Section IV.3]{natterer2001mathematics}, $s_j = c_R j^{-\frac{1}{2}}$, hence to get $\| (I - \mathbb{P}_q )A \|_{X\rightarrow Y} \leq c_R \delta$ it is enough to consider $q \geq \frac{1}{\delta^2}-1$
\end{rem}
\par
A well-know technique for the solution of the minimization problem \eqref{eq:minN} is the Iterative Soft Thresholding Algorithm (introduced in \cite{D}), which consists in selecting an initial guess $w^{(0)} \in \R^p (\cong W_p)$ and in iteratively creating the sequence $\{w^{(n)}\}_{n=1}^N$ as follows:
\begin{equation} \label{eq:ISTA}
    w^{(n)} = \mathcal{T}(w^{(n-1)}) = \mathcal{S}_{\lambda/L}\left(w^{(n-1)} - \frac{1}{L} W A_\vN^* A_\vN W^* w^{(n-1)} + \frac{1}{L} W A_{\vN}^* m\right),
\end{equation}
where $\frac{1}{L}>0$ is interpreted as a (fictitious) time step and, for $\beta>0$, $S_\beta(w)$ is the (component-wise) soft-thresholding operator:
\[ [\mathcal{S}_\beta(w)]_I = S_\beta(w_I); \qquad S_\beta(w_I) = \left\{ \begin{aligned}
    w_I + \beta \quad & \text{if } w_I < -\beta \\
    0 \quad & \text{if } |w_I| \leq \beta \\
    w_I - \beta \quad & \text{if } w_I > \beta \\
\end{aligned}\right. .\]
The convergence of $\{w^{(N)}\}$ to a minimizer $w_{\delta,\vN}$ of $\eqref{eq:minN}$ is analyzed, in an infinite dimensional context, in \cite{BL}. The following result for the discrete problem under consideration is instead a direct consequence of \cite[Theorem 25]{BNPS}:
\begin{proposition}
If $L$ is chosen such that $L \geq \| W A_\vN^* A_\vN W^* \|/2 $ then the sequence $\{w^{(N)}\}$ generated via \eqref{eq:ISTA} by any $w^{(0)} \in \R^p$ converges in $\ell^2$ to the solution $w_{\delta,\vN}$ of \eqref{eq:minN}. Moreover, there exist $c_3 > 0$ and $0 \leq a < 1$ (both depending on $A_\vN$, $L$ and $\| w^\dag\|_{\ell^2}$) such that 
\begin{equation} \label{eq:BL}
\| w^{(N)} - w_{\delta,\vN} \|_{\ell^2} \leq c_3 a^N.
\end{equation}
\end{proposition}

\subsection{A modification of ISTA}
\label{subsec:ista_perturbed}

We now consider a perturbation of ISTA \eqref{eq:ISTA}. Let $Z:\ell^2(\N) \rightarrow \ell^2(\N)$ satisfy
\begin{equation}
\| W A_\vN^* A_\vN W^* - Z \|_{\ell^2 \rightarrow \ell^2} \leq \rho.
    \label{eq:AZ}
\end{equation}
Then, we substitute $Z$ in place of the matrix $W A_\vN^* A_\vN W^*$ in the expression of ISTA. To remark the dependency on the perturbation amplitude $\rho$, we denote by $\{\wt{n}\}$ the sequence obtained by selecting $\wt{0} \in \R^p$ and iterating
\begin{equation}
  \wt{n} = \mathcal{T_Z}(\wt{n-1}) = \mathcal{S}_{\lambda/L}\left(\wt{n-1} - \frac{1}{L} Z \wt{n-1} + \frac{1}{L} W A_\vN^* m \right).
\label{eq:ISTAZ}    
\end{equation}
The following result shows a connection between the convergence of the sequence $\{\wt{n}\}$ to the minimizer $w_{\delta,\vN}$ and the magnitude of the perturbation $\rho$.
\begin{proposition} \label{prop:conv3}
Let $w^{(0)} = \wt{0}$, $L \geq \| W A_\vN^* A_\vN W^* \|$ and consider $N_0, \eta_0>0$. Then there exists a constant  $\tilde{c}_4$, depending on $L,A,w^{(0)}, \| w^\dag \|_{\ell^2}$ and on $N_0,\eta_0$, such that if $N \geq N_0$ and $\rho N \leq \eta_0$ then 
\begin{equation} \label{eq:conv4Old}
 \| \wt{N}-w_{\delta,\vN}\|_{\ell^2}  \leq c_3 a^N + \tilde{c}_4 \rho N.
\end{equation}
If, moreover, $N,\rho$ are chosen as $N > \frac{\ln(\delta^{-1})}{\ln(a^{-1})}$ and $\rho < \frac{\delta}{N}$, then (for $c_4 = c_3 + \tilde{c}_4$)
\begin{equation} \label{eq:conv4}
 \| \wt{N}-w_{\delta,\vN}\|_{\ell^2}  \leq c_4 \delta.
\end{equation}
\end{proposition}
The proof of this proposition follows by the nonexpansivity of the soft-thresholding operator and is reported in \cref{Appendix2}.

We collect the results obtained in \cref{prop:conv2} and \cref{prop:conv3} in the following final convergence estimate:
\begin{theorem}
 Let $w^\dag$ satisfy \eqref{eq:sparse} and let $A$ be injective. For sufficiently small $\delta$, select a regularization parameter $\lambda = c_0 \delta$. Select $p,q$ s.t. $\| w^\dag - \mathbb{P}_p w^\dag \| \leq c_p \delta$ and $\| (I - \mathbb{P}_q)A \|_{X \rightarrow Y} \leq c_q \delta$. Let $L \geq \| W A_\vN^* A_\vN W^* \| $ and consider the perturbed ISTA iterations \eqref{eq:ISTAZ}, where the operator $Z$ satisfies \eqref{eq:AZ}, $N = \log_a \delta$ and $\rho =  \frac{\delta}{N}$. Then, there exists a positive constant $c_5$ (depending on the previously introduced constants $c_0,c_1,c_2,c_3,c_4,c_p,c_q$) such that, for sufficiently small $\delta$,
 \begin{equation}
  \| \wt{N} - w^\dag \|_{\ell^2}  \leq c_5 \delta.
     \label{eq:superconv}
 \end{equation}
 \label{thm:superconv}
 \end{theorem}

\subsection{Wavelets in 2D}
\label{subsec:wavelets}

In order to derive the main results of the paper, we need to assume that the orthogonal basis $\{\psi_I\}_{I = 1}^\infty$ is a wavelet basis in $X = L^2(\Omega)$. Although our approach is sufficiently general to handle higher-dimensional spaces, we are going to focus on the two-dimensional case, \textit{i.e}., $\Omega \subset \R^2$ (\textit{e.g.}, $\Omega =  [0,1]^2$). Before moving to the representation of the operator $A^*A$ with respect to such basis, we need to describe in more details its structure.  \par
A common way to define a wavelet basis in $\R^2$ is to rely on two real functions $\psi$ and $\varphi$, respectively defined as mother wavelet and scaling function, whose support is in $[0,1]$. We identify an element $\psi_I$ of the basis by its scale $j$, its translation $k \in \N_0^2$ and its type $(t) \in \{(v),(h),(d),(f)\}$ (respectively, vertical, horizontal, diagonal and low-pass filter). We denote $\psi_I(x)$ as $\psi_{j,k}^{(t)}(x) = 2^j \psi^{(t)}(2^{j}x - k)$, $x \in [0,1]^2$, where we have:
\[
\begin{aligned}
    &\psi^{(v)}(x_1,x_2) = \phi(x_1)\psi(x_2) \qquad &\psi^{(h)}(x_1,x_2) = \psi(x_1)\phi(x_2) \\
    &\psi^{(d)}(x_1,x_2) = \psi(x_1)\psi(x_2) \qquad &\psi^{(f)}(x_1,x_2) = \phi(x_1)\phi(x_2) \\
\end{aligned}
\]
When selecting a maximum scale $J$ (and $J_0 < J$ as coarsest scale), we can define a wavelet basis of $p = 2^{2J}$ elements as follows: take $j \in \{J_0, \ldots, J_1 = J-1\}$; for each $j \neq J_0$, consider wavelets of the types $(v)$, $(h)$ and $(d)$, whereas for $j = J_0$ include also the type $(f)$. For each level $j$ and type $(t)$, consider offsets $k = (k_1, k_2)$, $k_1 = 0, \ldots, 2^j-1$, $k_2 = 0, \ldots, 2^j-1$. \par
We group the wavelet basis functions in subbands, each of which is identified by a scale $j$ and a type $(t)$, obtaining $3(J - J_0) + 1$ subsets.

\section{ISTA and Convolutional Neural Networks}
\label{sec:unrolled_ista_theo}

It is already well known that the unrolled iterations of ISTA can be considered as the layers of a neural network (see, \textit{e.g.}, \cite{Gregor}). Indeed, the $n$-th iteration of ISTA can be written as 
\begin{equation}
w^{(n)} = \mathcal{S}_{\lambda/L}\left(w^{(n-1)} - \frac{1}{L} K^{(n)} w^{(n-1)} + \frac{1}{L} b^{(n)} \right),
\label{eq:RESnet}    
\end{equation}
being $K^{(n)} = W A_\vN^* A_\vN W^*$ and $b^{(n)} = W A_\vN^* m$, independently of $n$. At the same time, \eqref{eq:RESnet} can be seen as the $n$-th layer of a recurrent Neural Network, where $K^{(n)}$ is the matrix of the weight coefficients and $b^{(n)}$ is the bias vector. When considering only the first $N$ iterations of ISTA, we can collect the parameters appearing in the layers in a vector $\theta \in \Theta$. Together with the entries of the matrices $K^{(n)}$, we may consider as parameters the steplength $L$ as well as the regularization parameter $\lambda$: see \cref{sec:our_models} for more details. Conversely, the bias vectors $b^{(n)}$ are not to be considered as parameters: they are fixed and equal to $W A_\vN^* m$ in each layer. We then introduce the map $f_\theta: Y \rightarrow \ell^1(\N)$, parameterized by $\theta \in \Theta$, which takes as an input $m \in Y_q$ and computes $N$ iterations like \eqref{eq:RESnet}, where, for each $n$, $K^{(n)} \in \R^{p \times p}$ is specified in $\theta$ and $b^{(n)} = W A_\vN^* m$. For any selected value of $p,q,N,\lambda,L$, we know that there exists a particular choice $\theta_0$ which corresponds to the ISTA iterations associated to the measurements $m$.
\par
In this section we show that, under some assumptions on the operator $A$, it is possible to interpret the operations in \eqref{eq:RESnet} as a layer of a CNN. We therefore provide a fairly general network architecture which allows to recover the standard ISTA iterations (or a perturbation of the kind described by \eqref{eq:AZ}), for a specific choice of the parameters.
\par
From now on, we focus on the case $X = L^2(\Omega)$, and consider a wavelet basis $\{\psi_I\}$ of the kind described in  \cref{subsec:wavelets}.

\subsection{A convolutional interpretation of ISTA}
\label{subsec:ista_as_cnn}

We first show, under additional assumptions on operator $A$, how to translate the Neural Network encoded by the operator $f_\theta$ above into a CNN, allowing for a significant reduction of the number of parameters involved.  
\par
Suppose that the operator $A^*A$ is a convolutional kernel operator, \textit{i.e.}, 
\begin{equation}
\begin{aligned}
 \mathbf{K}_{I,I'} = (A^*A \psi_I, \psi_{I'})_X &= \int_{\R^2} \int_{\R^2} K(x,x')\psi_I(x)\psi_{I'}(x')dx dx', \\
 K(x,x') &= K(x-x').
\end{aligned}
   \label{eq:convol}
\end{equation}
According to the description in \cref{subsec:wavelets}, the wavelet basis can be naturally split in subbands, each of which is identified by a couple $j$,$(t)$. This implies that the matrix $\mathbf{K}$ representing $A^*A$ can be seen as a block matrix. We now aim at describing the application of each block $\mathbf{K}_{j \rightarrow j'}^{(t) \rightarrow (t')} w_j^{(t)}$ by means of the following operations:
\begin{enumerate}
    \item Discrete convolution. Let $B \in \R^{b \times b}$, $C \in \R^{(2b-1)\times(2b-1)}$, and denote the elements of $C$ with indices $i,j$ , being $i = -b+1, \ldots, 0, \ldots, b-1$, $j = -b+1, \ldots, 0, \ldots, b-1$. Then, $C \ast B \in \R^{b \times b}$:
    \begin{equation}
        (C \ast B)_{k,l} = \sum_{i=0}^{b-1} \sum_{j=0}^{b-1} C_{k-i,l-j} B_{i,j}
        \label{eq:convolution}
    \end{equation}
    \item Upsampling. Let $B \in \R^{b \times b}$; then, $\mathscr{U}(B) \in \R^{2b \times 2b}$ satisfies:
    \begin{equation}
        \mathscr{U}(B)[2k:2k+1,2l:2l+1] = 
        \begin{bmatrix} 
            B_{k,l} & 0 \\ 0 & 0
        \end{bmatrix}
        \quad \forall k,l = 0,\ldots,b-1,
    \label{eq:upsample}
    \end{equation} 
    where the notation $\mathscr{U}(B)[2k:2k+1,2l:2l+1]$ is used to denote a submatrix of $\mathscr{U}(B)$ containing the rows from $2k$ to $2k+1$ and all the columns from $2l$ to $2l+1$. We denote by $\mathscr{U}^\eta$ the iterated application of $\mathscr{U}$: $\mathscr{U}^\eta = \mathscr{U} \circ \ldots \circ \mathscr{U} $ ($\eta$ times).
    \item Downsampling. Let $B \in \R^{2b \times 2b}$; then, $\mathscr{D}(B) \in \R^{b \times b}$ satisfies:
    \begin{equation}
        \mathscr{D}(B)_{k,l} = B_{2k,2l}
        \quad \forall k,l = 0,\ldots,b-1.
    \label{eq:downsample}
    \end{equation}
     We denote by $\mathscr{D}^\eta$ the iterated application of $\mathscr{D}$: $\mathscr{D}^\eta = \mathscr{D} \circ \ldots \circ \mathscr{D} $ ($\eta$ times).
\end{enumerate}
The following crucial result provides a full description of the convolutional interpretation of the matrix representing $A^*A$ in the wavelet domain. Such result can be compared to the ones already known in literature, see \textit{e.g.} \cite[Formula (4.2)]{dahmen1994wavelet}, although the more complicated structure of the wavelet basis entails some significant differences.
\begin{proposition}
Let $\mathbf{K} \in \R^{p \times p}$ be the matrix representing an operator $A^*A$ satisfying \eqref{eq:convol} in a 2D wavelet basis $\{\psi_I\}_{I = 1}^p$. For a vector $w \in \R^p$, let $w_j^{(t)}$ be the vector of the wavelet components related to basis functions of scale $j$ and type $(t)$. Let $\mathbf{K}_{j \rightarrow j'}^{(t) \rightarrow (t')}$ denote the block of $\mathbf{K}$ corresponding to the $j,(t)$ subset of the column indices and the $j',(t')$ subset of the row indices. Then
\begin{equation}
    \mathbf{K}_{j \rightarrow j'}^{(t) \rightarrow (t')} w_j^{(t)} = \left\{ 
    \begin{aligned}
     \mathscr{D}^{\delta} (& \widetilde{\mathbf{K}}_{j \rightarrow j'}^{(t) \rightarrow (t')} \ast  W_j^{(t)}) & \qquad \text{if} \quad j < j' \\
    & \widetilde{\mathbf{K}}_{j \rightarrow j'}^{(t) \rightarrow (t')} \ast W_j^{(t)} & \qquad \text{if} \quad j = j'\\
    & \widetilde{\mathbf{K}}_{j \rightarrow j'}^{(t) \rightarrow (t')} \ast \mathscr{U}^{\delta}( W_j^{(t)}) & \qquad \text{if} \quad j > j'
\end{aligned}
 \right.
\label{eq:madness}    
\end{equation}
being $\delta = |j'-j|$, and  $\widetilde{\mathbf{K}}_{j \rightarrow j'}^{(t) \rightarrow (t')} \in \R^{(2^{\hat{j}+1} - 1)\times(2^{\hat{j}+1} - 1)}$ (where $\hat{j} = \max(j,j')$):
\begin{equation}
\begin{aligned}
    \left[\widetilde{\mathbf{K}}_{j \rightarrow j'}^{(t) \rightarrow (t')}\right]_d = \int_{\R^2}\int_{\R^2} K(x-x'-2^{-\hat{j}}d) \ \psi_{j',0}^{(t')}(x') \ \psi_{j,0}^{(t)}(x) dx dx' \\
    d = (d_1,d_2); \quad d_1,d_2 = \{-2^{\hat{j}}+1,\ldots,0,\ldots,2^{\hat{j}}-1\}.
\end{aligned}
\label{eq:madness2}    
\end{equation}
The matrix $W_j^{(t)} \in \R^{2^j \times 2^j}$ is obtained by reshaping the vector $w_j^{(t)} \in \R^{2^{2j}}$ so that $[W_j^{(t)}]_d$ is the component $w_I$ whose index is identified by $(j,(t),d)$.
\label{prop:convolution}
\end{proposition}
\begin{proof}
Let $I,I'$ be identified by $(j,(t),k)$ and $(j',(t'),k')$, respectively. Then,
\[
\begin{aligned}
    \left[\mathbf{K}\right]_{I',I} & = \int_{\R^2}\int_{\R^2} K(x-x') \ \psi_{j,k'}^{(t')}(x') \ \psi_{j,k}^{(t)}(x) dx dx' \\
    & = \int_{\R^2}\int_{\R^2} K(x-x') \ \psi_{j,0}^{(t')}(x'-2^{-j'}k') \ \psi_{j,0}^{(t)}(x-2^{-j}k) dx dx' \\
    & = \int_{\R^2}\int_{\R^2} K(x+2^{-j}k-x'-2^{-j'}k') \psi_{j,0}^{(t')}(x) \psi_{j,0}^{(t)}(x) dx dx' \\
    & = \int_{\R^2}\int_{\R^2} K(x-x'-2^{-\hat{j}}(2^{\delta^-} k' - 2^{\delta^+} k)) \psi_{j,0}^{(t')}(x) \psi_{j,0}^{(t)}(x) dx dx' \\
    & = \left[\widetilde{\mathbf{K}}_{j \rightarrow j'}^{(t) \rightarrow (t')}\right]_{d},
\end{aligned}
\]
where $\delta^+ = \max(0,j-j')$, $\delta^- = \max(0,j'-j)$, and $d = 2^{\delta^-} k' - 2^{\delta^+} k$. For the sake of ease, we use $\mathbf{K}$ instead of $\mathbf{K}_{j \rightarrow j'}^{(t) \rightarrow (t')}$, $\widetilde{\mathbf{K}} $ instead of $\widetilde{\mathbf{K}}_{j \rightarrow j'}^{(t) \rightarrow (t')}$, $w$ instead of $w_j^{(t)}$, $W$ instead of $W_{j}^{(t)}$. Moreover, we denote by $\mathcal{I}$ the set of indices $\mathcal{I}\subset \{1,\ldots,p\}$ belonging to the wavelet scale $j$ and type $(t)$.
\par
Consider first the case $j = j'$. Then $\delta = \delta^+ = \delta^- = 0$, and it holds
\[ 
\left[\mathbf{K}\right]_{I',I} = \left[\widetilde{\mathbf{K}} \right]_{d}, \qquad d = k' - k.
\]
Therefore,
\[
\begin{aligned}
\left[ \mathbf{K} w \right]_{I'} &= \sum_{I \in \mathcal{I}} \left[\mathbf{K}\right]_{I',I} w_I = \sum_{I \in \mathcal{I}} \left[\widetilde{\mathbf{K}} \right]_{k' - k(I)} w_I \\
&= \sum_{k_1 = -2^{j}}^{2^{j}}\sum_{k_2 = -2^{j}}^{2^{j}} \left[\widetilde{\mathbf{K}} \right]_{k'_1 - k_1, k'_2-k_2} W_{k_1,k_2} = [\mathbf{K} \ast W]_{I'}. 
\end{aligned}
\]
Let now $j < j'$. Then $\delta = \delta^+ >0$, $\delta^- = 0$, and
\[
\begin{aligned}
\left[ \mathbf{K} w \right]_{I'} &= \sum_{I \in \mathcal{I}} \left[\mathbf{K}\right]_{I',I} w_I = \sum_{I \in \mathcal{I}} \left[\widetilde{\mathbf{K}} \right]_{k' - 2^{\delta^+}k(I)} w_I \\
&= \sum_{k_1 = -2^{j'}}^{2^{j'}}\sum_{k_2 = -2^{j'}}^{2^{j'}} \left[\widetilde{\mathbf{K}} \right]_{k'_1 - 2^{\delta^+} k_1, k'_2- 2^{\delta^+}k_2} \mathscr{U}^{\delta^+}(W)_{2^{\delta^+}k_1,2^{\delta^+}k_2} = [\mathbf{K} \ast \mathscr{U}^\delta W]_{I'}. 
\end{aligned}
\]
Finally, let $j > j'$. Then $\delta^+ = 0$, $\delta = \delta^- > 0$, and
\[
\begin{aligned}
\left[ \mathbf{K} w \right]_{I'} &= \sum_{I \in \mathcal{I}} \left[\mathbf{K}\right]_{I',I} w_I = \sum_{I \in \mathcal{I}} \left[\widetilde{\mathbf{K}} \right]_{2^{\delta^-}k' - k(I)} w_I \\
&= \sum_{k_1 = -2^{j}}^{2^{j}}\sum_{k_2 = -2^{j}}^{2^{j}} \left[\widetilde{\mathbf{K}} \right]_{2^{\delta^-} k'_1 - k_1, 2^{\delta^-}k'_2 - k_2} W_{k_1,k_2} = [\mathscr{D}^{\delta}(\mathbf{K} \ast W)]_{I'}. 
\end{aligned}
\]
\end{proof}
\begin{rem}
 The most relevant consequence of \cref{prop:convolution} is a significant reduction of the coefficients required to describe the application of $A^*A$ as a function from $\R^p$ to $\R^p$. The standard representation, obtained by a matrix in $\R^{p \times p}$, involves indeed $p^2 = 2^{4J}$ parameters, whereas the representation via the convolutional filters $\widetilde{\mathcal{K}}_{j \rightarrow j'}^{(t) \rightarrow (t')}$ involves only $O(p)$ elements.
\end{rem}

This convolutional interpretation also reflects on the Neural Network architecture proposed in \eqref{eq:RESnet}: if we substitute the multiplication $K^{(n)} w^{(n-1)}$ by the operations encoded by \eqref{eq:madness} (decomposition of $w^{(n-1)}$ in wavelet subbands, upscaling, application of convolutional filters, downscaling), the parameters $\theta$ involved in the description of $K^{(n)}$ are reduced.
The representation of the linear operators $K^{(n)}$ through convolutions, upscaling and downscaling is a typical feature of CNNs: thus, by designing a CNN which reproduces exactly the operations reported in \eqref{eq:madness} and \eqref{eq:RESnet}, we can ensure that such a network is completely equivalent, for a suitable choice $\theta_0$ of the parameters, to the application of ISTA.

\subsection{A working example}
\label{subsec:example}

In order to better visualize the convolutional representation of ISTA reported in \eqref{eq:madness2}, we now provide a small example. Consider the case of $64 \times 64$ images, thus associated to $J = 6$ and $p = 2^{12}$. Create a wavelet basis consisting of three scales of wavelets, from $J_0 = 3$ to $J_1 = 5$. The resulting basis $\{\psi_I\}_{I=1}^p$ can be therefore split into $10$ subbands: $4$ associated to the scale $j=3$ (types: $(h)$, $(v)$, $(d)$ and $(f)$), $3$ associated to the scales $j=4$ (types: $(h)$, $(v)$, $(d)$) and $3$ with $j = 5$. Each subband consists of $2^{2j}$ elements. \\
The operator $A^*A$ is represented in the wavelet basis $\{\psi_I\}$ by means of a matrix $\mathbf{K} \in \R^{p \times p}$. According to \cref{subsec:ista_as_cnn}, the following procedure is equivalent to apply the matrix $\mathbf{K}$ on a vector $w \in \R^p$ (representing the wavelet transform of an image):
\begin{enumerate}
    \item First, split the vector $w$ into its $10$ wavelet subbands, each of which identified by a scale $j$ and a type $(t)$. This operation is depicted in \cref{fig:example1}. The vector $w_j^{(t)} \in \R^{2j}$ can also be interpreted as a matrix $W_j^{(t)} \in \R^{j \times j}$. The element $[W_j^{(t)}]_{d} = [W_j^{(t)}]_{(d_1,d_2)}$ is the component associated to the basis function $\psi_{j,d}^{(t)}(x) = 2^j \psi^{(t)}(2^j x_1 -d_1,2^j x_2 - d_2)$.
    \begin{figure}
    \centering
    \includegraphics[width=0.75\textwidth]{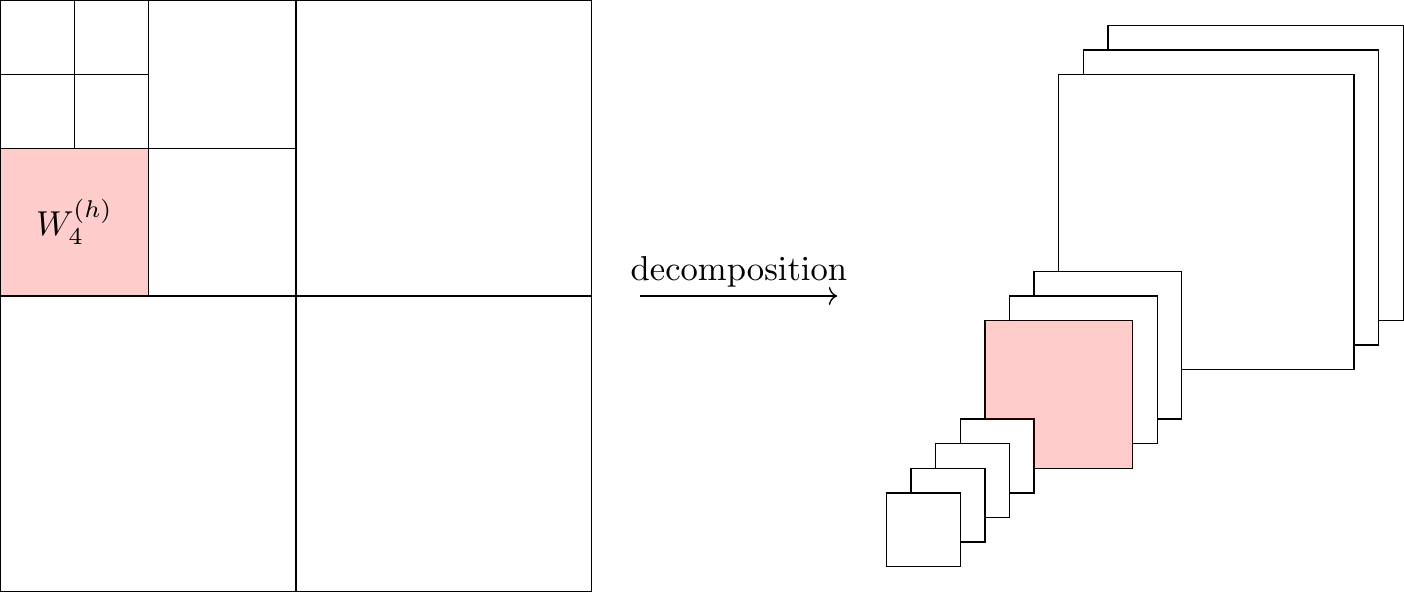}
    \caption{Interpretation of \eqref{eq:madness}. \\ Step 1: decompose the wavelet transform in subbands}
    \label{fig:example1}
\end{figure}
    \item Secondly, for each subband $j,(t)$, compute the $10$ vectors $\mathbf{K}_{j \rightarrow j'}^{(t) \rightarrow (t')} w_j^{(t)}$, the contributions of $w_j^{(t)}$ on the subband $j',(t')$ of the vector $\mathbf{K} w$. Each matrix $\mathbf{K}_{j \rightarrow j'}^{(t) \rightarrow (t')}$ is a $2^{2j'} \times 2^{2j}$ block composing the matrix $\mathbf{K}$. According to \eqref{eq:madness}, this can be done by means of usampling, downsampling and convolution. Consider the case $j =J_0 = 3$:
    \begin{itemize}
        \item if $j'=3$, then $\hat{j}= 3$ and $\delta = 0$. Thus, if we compute the convolution of the $15\times 15$ filter $\mathbf{\widetilde{K}}_{3 \rightarrow 3}^{(t) \rightarrow (t')}$ with the matrix $W_3^{(t)} \in \R^{8 \times 8}$, we get a $8 \times 8$ matrix representing the vector $\mathbf{K}_{3 \rightarrow 3}^{(t) \rightarrow (t')} w_3^{(t)} \in \R^{64}$ . 
        \item if $j'=4$, then we shall use the first variant in formula \eqref{eq:madness} with $\delta = 1$ (whereas in \eqref{eq:madness2} we have $\hat{j}=4$). To compute the $16 \times 16$ matrix associated to $\mathbf{K}_{3 \rightarrow 4}^{(t) \rightarrow (t')} w_3^{(t)}$, we must first upsample the matrix $W_3^{(t)}$ and then convolve it with the $31\times 31$ filter $\mathbf{\widetilde{K}}_{3 \rightarrow 4}^{(t) \rightarrow (t')}$.
        \item if $j' = 5$, then we use again the first variant of \eqref{eq:madness}, with $\delta = 2$; hence the matrix $W_3^{(t)}$ must be upsampled twice before being convolved with the $63 \times 63$ filter $\mathbf{\widetilde{K}}_{3 \rightarrow 5}^{(t) \rightarrow (t')}$.
    \end{itemize}
     Consider instead the case $j = 4$:
    \begin{itemize}
        \item if $j'=3$, then we need to use the third variant in \eqref{eq:madness} with $\delta = 1$ (and \eqref{eq:madness2} with $\hat{j}= 4$), which means we first compute the convolution between the $31\times 31$ filter $\mathbf{\widetilde{K}}_{4 \rightarrow 3}^{(t) \rightarrow (t')}$ and the matrix $W_4^{(t)} \in \R^{16 \times 16}$ and then to downscale it to recover the $8 \times 8$ matrix describing $\mathbf{K}_{4 \rightarrow 3}^{(t) \rightarrow (t')} w_4^{(t)}$.
        \item the case $j'=4$ is analogous to the $3 \rightarrow 3$ one, using $31 \times 31$ filters  $\mathbf{\widetilde{K}}_{4 \rightarrow 4}^{(t) \rightarrow (t')}$.
        \item the case $j'= 5$ is analogous to the $3 \rightarrow 4$ one: we first perform upsampling and then convolution.
    \end{itemize}
    Finally, for $j=J_1 = 5$,
    \begin{itemize}
        \item if $j'=3$, then we first compute the convolution between $\mathbf{\widetilde{K}}_{5 \rightarrow 3}^{(t) \rightarrow (t')} \in \R^{63 \times 63}$ and $W_5^{(t)} \in \R^{32 \times 32}$ and then downsample twice.
        \item if $j'=4$ we only downsample once, as in the case $4 \rightarrow 3$.
        \item if $j'= 5$, we only do convolution, as in the cases $3 \rightarrow 3$ and $4 \rightarrow 4$, but with $63 \times 63$ filters.
    \end{itemize}
    A graphical visualization of these operations is provided by  \cref{fig:example2}.
    \begin{figure}
    \centering
    \includegraphics[width=0.75\textwidth]{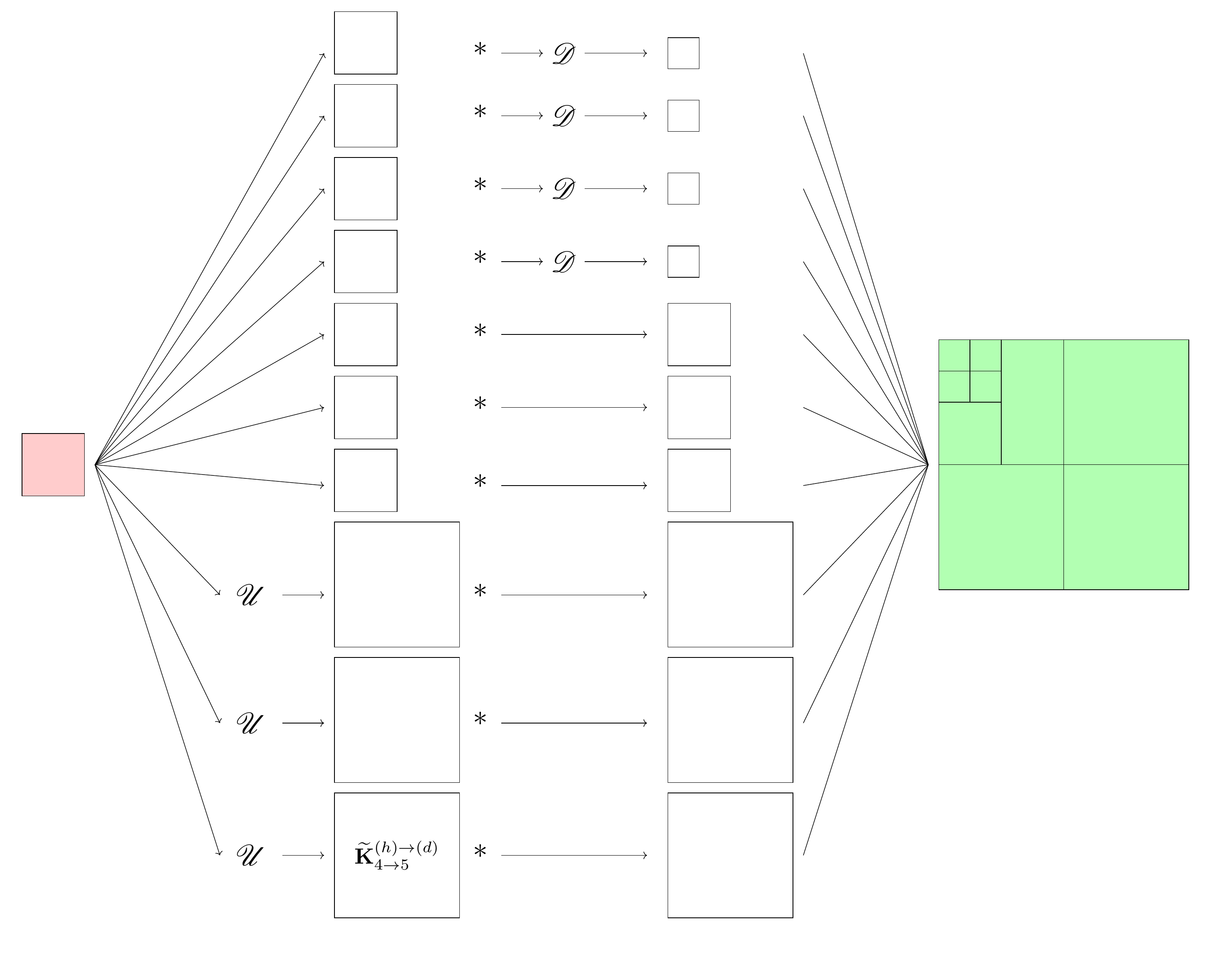}
    \caption{Interpretation of \eqref{eq:madness}. \\ Step 2: convolution, upsampling and downsampling}
    \label{fig:example2}
\end{figure}
\item The last step consists of collecting, for each subband $j',(t')$, all the contributions coming from the vectors $w_j^{(t)}$. Thanks to the previous step, among the $100$ computed matrices, all the ones associated to those contributions have dimensions $2^{j'} \times 2^{j'}$. By adding them up we recover the $j',(t')$ subband of the vector $\mathbf{K}w$ (see \cref{fig:example3}).
    \begin{figure}
    \centering
    \includegraphics[width=0.75\textwidth]{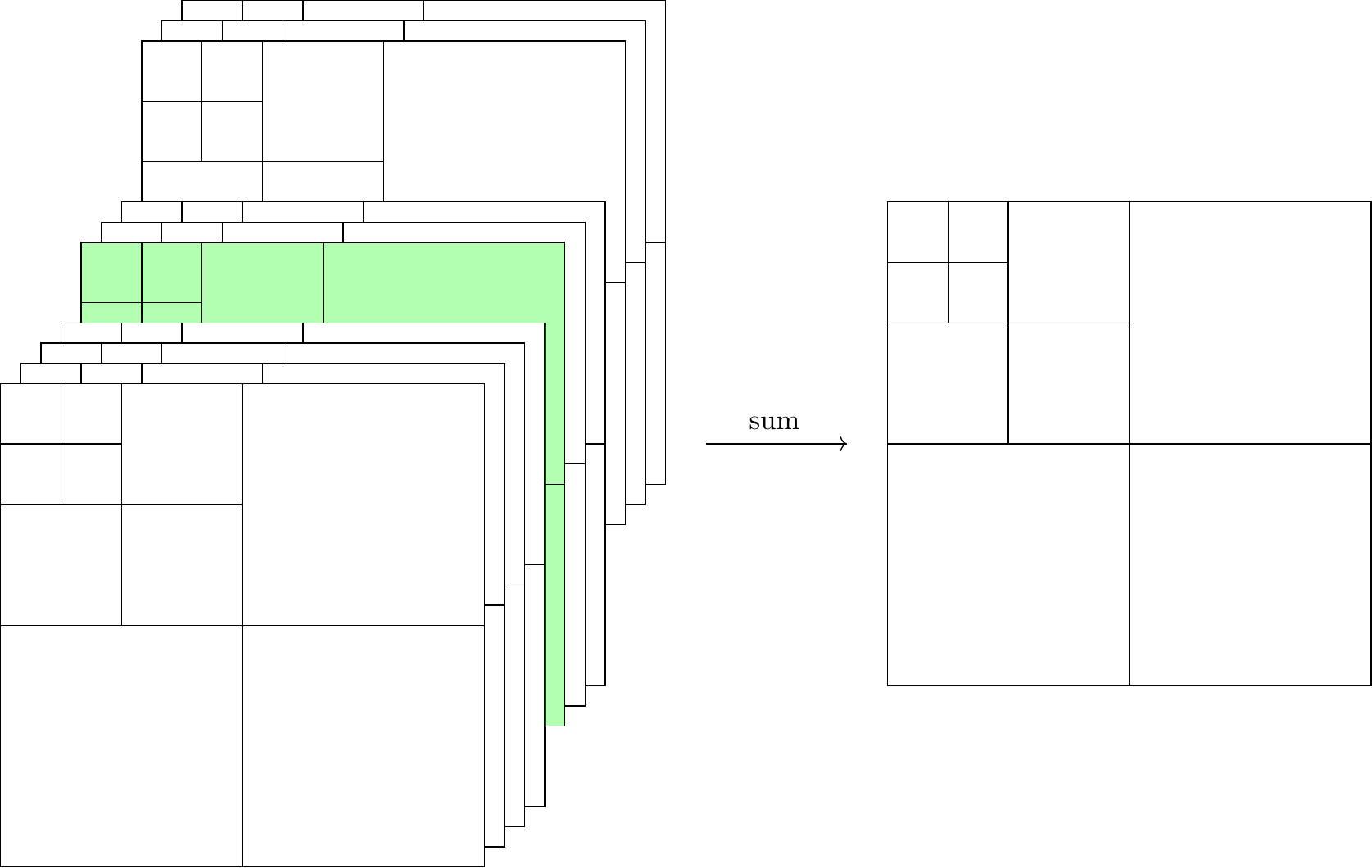}
    \caption{Interpretation of \eqref{eq:madness}. \\ Step 3: reassembling each wavelet subband.}
    \label{fig:example3}
    \end{figure}
\end{enumerate}

\subsection{On the possibility to use smaller filters}
\label{subsec:small_filters}

When designing a CNN, it is common practice to employ a large numbers of convolutional filters of small size. In the architecture determined by \eqref{eq:madness} and \eqref{eq:madness2}, the required number of filters is exactly $(3(J - J_0) + 1)^2$, and each part of the vector $w^{(n-1)}$ interacts only with $(3(J - J_0) + 1)$ of them. Moreover, the size of each filter must be equal to $(2^{j'+1}-1)(2^{j+1}-1)$. We now consider the effect of substituting such large filters with smaller ones.
\par
We would like to use filters of size $\tau \times \tau$, being $\tau = (2\xi + 1)$ and $\xi > 1$, obtained by extracting the central elements of the large filters $\widetilde{\mathbf{K}}_{j \rightarrow j'}^{(t) \rightarrow (t')}$. In particular, we define $\widetilde{\mathbf{K}}^\tau = (\widetilde{\mathbf{K}}_{j \rightarrow j'}^{(t) \rightarrow (t')})_\tau$, being $\tau = 2\xi + 1$, as 
\begin{equation}
    \left[\widetilde{\mathbf{K}}^\tau\right]_d = \left\{ 
    \begin{aligned}
    & \left[\widetilde{\mathbf{K}}_{j \rightarrow j'}^{(t) \rightarrow (t')}\right]_d \qquad &if\  \| d \|_{\infty} \leq \xi, \\
    & \ 0 \qquad &if\  \| d \|_{\infty} > \xi.
    \end{aligned}
    \right.
    \label{eq:gammathresh}
\end{equation}
We claim that this modification is equivalent to performing a perturbation of ISTA of the type treated in \cref{prop:conv3}, where the parameter $\rho$ is a suitable function of $\tau$. Although providing a detailed proof of this would entail cumbersome computation, we prove the most important result which is required to accomplish this task: we exhibit a bound on the coefficients of the filters which are discarded due to \eqref{eq:gammathresh}.
\par 
Such an estimate can be obtained by assuming further hypotheses on the operator $A$. In particular, suppose that $A^*A$ is a convolutional operator of kernel $K$ (as in \eqref{eq:convol}) and, in addition, that for $x\neq x'$ the kernel $K(x,x')$ is smooth and such that
\begin{equation}
K(x,x') \leq \frac{C}{|x-x'|} \qquad |\nabla_x K(x,x')| + |\nabla_{x'} K(x,x')| \leq \frac{C}{|x-x'|^2}.
    \label{eq:decayK}
\end{equation}
It is easy to verify that \eqref{eq:decay} is satisfied whenever $A^*A$ is a $\Psi$DO of order $-1$ with constant coefficients, that is
\[
A^*A f = \mathcal{F}^{-1}\left\{ a(\xi)  \mathcal{F} \left\{ f \right\}(\xi) \right\}, \qquad a(\xi) \sim \frac{1}{|\xi|} \ as\  \xi \rightarrow 0.
\]
Moreover, we need to assume a property related to the wavelet basis functions, known as first-order vanishing moment:
\begin{equation}
\int_{\R^2} \psi_I(x) dx = 0.
\label{eq:moment}
\end{equation}
Such property is verified even by the 2D Haar wavelet system, apart from the type $(f)$.
\begin{proposition}
Let the operator $A$ satisfy \eqref{eq:convol} and \eqref{eq:decayK}. Let the indices $I,I'$ denote two wavelets of scales $j,j'$, type $(t),(t')$ and offsets $k,k'$. Let $\psi_I$ and $\psi_{I'}$ satisfy \eqref{eq:moment} and let $d_{I,I'}$ be the distance between the supports of $\psi_I$ and $\psi_{I'}$. Whenever $d_{I,I'}>0$, it holds:
\begin{equation}
\mathbf{K}_{I,I'} = (A^*A \psi_I,\psi_{I'})_X \leq c \frac{2^{-2(j+j')}}{d_{I,I'}^3}
    \label{eq:decay}
\end{equation}
\label{prop:decay}
\end{proposition}
We remark that the decay reported in \eqref{eq:decay} closely resembles formula (9.22) in \cite{dahmen1997wavelet} (according to the choice $n = 2$, $\tilde{d}=1$, $r = 2t = -1$) and with minor changes also formula (4.26) in \cite{beylkin1991fast} (with $M = 2$).
\begin{proof}
According to \eqref{eq:decayK}, and to \eqref{eq:moment}, 
for any choice of $x_0 \in \supp{\psi_{I}}$, $x_0' \in \supp{\psi_{I'}}$ there exists two points $\xi, \xi'$ in the same supports such that
\[
\begin{aligned}
\mathbf{K}_{I,I'} &= \int_{\R^2} \int_{\R^2} (K(x,x')-K(x,x_0'))\psi_I(x)\psi_{I'}(x')dx dx' \\
&\leq \int_{\R^2} \int_{\R^2} |\nabla_x K(x,\xi')| |x'-x_0'| \psi_I(x)\psi_{I'}(x')dx dx' \\
&\leq C \int_{\R^2} \int_{\R^2}\frac{|x'-x_0'|}{|x-\xi'|^2} \psi_I(x)\psi_{I'}(x')dx dx' \\
&\leq C \int_{\R^2} \int_{\R^2} \frac{|x-x_0||x'-x_0'|}{|\xi-\xi'|^3} \psi_I(x)\psi_{I'}(x')dx dx'. 
\end{aligned}
\]
The quantity $|\xi-\xi'|$ is bounded from below by $d_{I,I'}$ by definition. Moreover, $|x-x_0| \leq \diam(\supp{\psi_I}) = c 2^{-j}$, and finally $\int_{\R^2} \psi_{I}(x) \leq 2^{j} |\supp{\psi_I}| = 2^{-j}$ (analogous arguments hold on $I'$).
\end{proof}
In view of \eqref{eq:decay} and of \eqref{eq:madness2}, we can easily obtain a bound on the elements of the convolutional filters:
\[
\left[\widetilde{\mathbf{K}}_{j \rightarrow j'}^{(t) \rightarrow (t')}\right]_d \leq c \frac{2^{-\hat{j}}}{(\| d \|_{\infty}-1)^3},
\]
provided that $\| d \|_{\infty} > 1$. This result, together with \eqref{eq:madness}, allows to obtain an explicit bound (in the form of \eqref{eq:conv4}) on the perturbation induced by the thresholding \eqref{eq:gammathresh}.

\section{$\Psi$DONet: formulation and theoretical results}
\label{sec:algo_theo}

In this section we introduce a reconstruction algorithm for sparsity-promoting regularization based on CNNs, which leads to a novel network architecture defined $\Psi$DONet. We report the general idea inspiring such a technique, taking advantage of the theoretical results obtained in  \cref{sec:unrolled_ista_theo} and providing a comprehensive interpretation. Eventually, we provide a theoretical result ensuring the convergence of the proposed algorithm.

\subsection{$\Psi$DONet: a network to learn pseudodifferential operators}
\label{subsec:psidonet}

Inspired by the results of the \cref{sec:unrolled_ista_theo}, if the operator $A^*A$ is of convolutional type, we define a reconstruction algorithm by designing a CNN of $N$ layers, each of which is described by \eqref{eq:RESnet}. In particular, the bias vectors appearing in \eqref{eq:RESnet} are $b^{(n)} = W A_\vN^*m$ for each $n$, whereas the linear operators $K^{(n)}$ are interpreted as a combination of upscaling, downscaling and convolution as described in \eqref{eq:madness}. As shown in \cref{prop:convolution}, if the entries of the convolutional filters are selected as is \eqref{eq:madness2}, this procedure is equivalent to performing $N$ iterations of ISTA. Instead, the key idea of the proposed algorithm is to split the convolutional filters into two parts: a central $\tau \times \tau$ submatrix (where $\tau$ is a predefined hyper-parameter of the algorithm) and the outer frame. For each one of the $3(J-J_0)+1$ filters required by each layer, we suppose that the entries in the external frame are specified according to \eqref{eq:madness2}, whereas the central entries are considered as parameters, to be learned throughout the training process. Such parameters are collected in a vector $\theta_n$ (related to the $n$-th layer) and ultimately stored in the vector $\theta$, possibly together with other learnable parameters. The obtained network is denoted as $f_\theta^\tau$: the aim of a CNN-based algorithm is to find a parameter $\theta$ such that the network is a good approximation of the solution map of our inverse problem, taking as an input the measurements $m$ and giving as an output the solution $w^\dag = Wu^\dag$.

It is evident that, among the possible choices of the optimal parameter, the network could select the vector $\theta_0$ which exactly replicates the ISTA iterations (it is the one for which, in every layer, also the central entries of each filter are specified by \eqref{eq:madness2}). Nonetheless, if the optimal choice of $\theta$ differs from $\theta_0$, it means that the network is learning something more than the ISTA iterations associated to the operator $A^*A$. This can be meaningfully interpreted as follows: in each layer, the network $f_\theta^\tau$ applies the filters associated to an operator whose kernel is $K_0 + K_1$, where $K_0$ is the kernel of $A^*A$ and $K_1$ is the kernel of another, \textit{learned}, operator. Since the difference will only occur in the central elements of the convolutional filters, according to the analysis of \cref{subsec:small_filters}, we can argue that the learned operator is indeed a suitable approximation of a \textit{pseudodifferential} operator. This finally allows to motivate the name we propose for this novel CNN-based reconstruction algorithm: \textbf{$\Psi$DONet}. 

There are several reasons for which the learning process could attain a better result than the one provided by ISTA. Indeed, a better choice of the parameters allows to reduce numerical errors induced by the discrete representation of $A^*A$, which might have a significant effect due to the error propagation among the iterations. Moreover, we might also mitigate model errors in the definition of the operator $A$ itself. Finally, this perturbation could provide a representation of $A^*A$ with respect to a slightly different basis, which allows to better satisfies the sparsity assumption on the solutions. For such reasons, the use of $\Psi$DONet is specifically recommended whenever the original operator $A^*A$ is a $\Psi$DO itself. Indeed, its kernel representation by means of convolutional filters might benefit from learned corrections in all its most important entries: namely, the central ones.

We will show that $\Psi$DONet is also highly recommended for FIOs: in this case, the largest entries of the convolutional filters representing $A^*A$ are located in the center and along some lines, possibly stretching away from the center. This is the case of the limited-angle Radon transform (deeply analyzed in the following sections), which is associated with the kernel
\[
K(x,y) = \frac{1}{|x-y|}\chi_\Gamma(x-y), 
\]
being $\chi_\Gamma$ the indicator function of the cone in $\R^2$ between the angles $-\Gamma$ and $\Gamma$. As reported in  \cref{sec:in_practice}, the convolutional filters related to this operator show large values only in the central elements and along two lines having the same slope of the ones delimiting the cone. This provides a curious shape, allowing us to rename them as \textit{bowtie} filters. We will show that the application of $\Psi$DONet on this operator, providing learned corrections only to the center of the bowties, is still extremely effective. 

\subsection{A convergence result}
\label{subsec:cnn_converg}
We now provide a theoretical result which holds true for the $\Psi$DONet algorithm, regardless of its specific implementation. In analogy with the approach of \cite{de2019deep}, we introduce the following probabilistic approach. Let $\mathcal{B} = \{ u \in X_p: \ Wu \in \ell^1(\N); \| Wu \|_{\ell^1} \leq C_\mathcal{B} \}$ and $u$ be a random variable having a probability distribution $\mu$ on the space $\mathcal{B}$. We can consider $\mu$ as some prior information on the solution of the inverse problem.
Moreover, let $\epsilon$ be a random variable in $Y_q$ with distribution $\nu$, which models the error on the measurements. Assume that $u$ and $\epsilon$ are independent: hence, the measurement $m = A_\vN u + \epsilon$ is a random variable on the product space $X_p \times Y_q$ with density $A_*\mu \otimes \nu$, where $A_*\mu$ denotes the pushforward of $\mu$ to $Y$ via the linear map $A$. In order to measure the performance of the network $f_\theta^\tau$, we introduce the loss function associated to the network $f_\theta^\tau$ as:
\begin{equation}
\mathcal{L}(\theta;\mu,\nu) = \mathbb{E}_{u \sim \mu, \epsilon \sim \nu} \left[ \| f_\theta^\tau(A_\vN u + \epsilon) - Wu \|^2_{\ell^2} \right]
    \label{eq:loss}
\end{equation}
We define the optimal Neural Network as the one associated to $\theta^*$ satisfying:
\begin{equation}
\theta^* = \argmin_{\theta \in \Theta} \mathcal{L}(\theta; \mu, \nu).
    \label{eq:thetastar}
\end{equation}
Before focusing on the properties of the optimal network $f^\tau_{\theta^*}$, it is convenient to recall that, for a specific choice of parameters $\theta_0$, the network $f_{\theta_0}^\tau$ is equivalent to performing $N$ iterations of (modified) ISTA. The following rough estimate will be useful:
\begin{lemma}
There exist two constants $k_1, k_2 > 0$ (depending on $C_{\mathcal{B}}, L,\rho$, $\|A_\vN\|$, $w^{(0)}$,$N$) such that, for all $u \in \mathcal{B}$ and $\epsilon \in Y_q$, 
\begin{equation}
\| f_{\theta_0}^\tau(A_\vN u + \varepsilon) - W u\|_{\ell^2} \leq k_1 + k_2 \| \epsilon \|_{Y_q}
    \label{eq:rough}
\end{equation}
\label{lem:rough}
\end{lemma}
\begin{proof}
According to \eqref{eq:RESnet}, defining $\kappa = 1 + \frac{\| A_\vN^*A_\vN \|+\rho}{L}$, we get
\[
\begin{aligned}
    \| f_{\theta_0}^\tau(A_\vN u + \varepsilon) - W u\|_{\ell^2} & \leq \| f_{\theta_0}^\tau(A_\vN u + \varepsilon)\|_{\ell^2} + \| u \|_{X_p} \\
    & \leq \kappa^N \| w^{(0)} \| + \left( 1 + \kappa + \ldots + \kappa^{N-1}\right) \| A_\vN u + \epsilon\|_{Y_q} +  C_{\mathcal{B}}\\
    & \leq \kappa^N \| w^{(0)} \| +  C_{\mathcal{B}} + \frac{\kappa^N - 1}{\kappa - 1}(\|A_\vN\| C_\mathcal{B} + \| \epsilon \|_{Y_q} ) 
\end{aligned}
\]
\end{proof}
We now focus on the case in which $\epsilon$ is a Gaussian random vector, i.e., $\nu = N(0,\sigma^2 I_q)$, being $I_q$ the identity matrix in $\R^{q \times q}$. In this case, it is useful to recall that
\begin{equation}
    \mathbb{E}[\| \varepsilon \|_{Y_q}^2] = q \sigma^2, \qquad \mathbb{E}[\| \varepsilon \|_{Y_q}^4] \leq 3 q^2 \sigma^4.
    \label{eq:moments}
\end{equation}
In addition to \cref{lem:rough}, we can rely on the results reported in \cref{sec:theo_background} (and in particular on \cref{thm:superconv}) to provide a more refined estimate. Indeed, we observe that the convergence result reported in \eqref{eq:superconv} is independent of the choice of $\epsilon = m - A u^\dag$, as long as $\| \epsilon \| \leq \delta$. Moreover, the constant $c_5$ appearing in \eqref{eq:superconv} can depend on $u^\dag$, but only through an upper bound on $\|w^\dag\|_{\ell^1}$ (see, in particular, \cite[Theorem 1]{flemming2018injectivity} and to \cite[Theorem 25]{BNPS} for the constant derived from \cref{prop:conv2} and \cref{prop:conv3}, respectively). This allows us to conclude that
\begin{lemma}
Suppose $\epsilon \sim N(0,\sigma^2 I_q)$ and let $\delta = \sigma^{1/\eta}$, being $\eta>1$. There exists $\sigma_0 > 0$ such that, for $\sigma < \sigma_0$, then for every $u \in \mathcal{B}$
\[
\mathbb{E}_{\varepsilon \sim \nu}\left[\| f_{\theta_0}^\tau(Au + \epsilon) - Wu \|_{\ell^2}^2\right] \leq c_5^2 \delta^2 + 2\sqrt{2}k_1^2 \delta^{\eta - 1} + 2\sqrt{6}k_2^2 q \delta^{3\eta - 1}.
\]
If, moreover, $\eta = 3$ and $\sigma < \min\{\sigma_0,q^{-1/2}\}$,
then there exists a constant $c^{*}$ (depending on $c_5, k_1,k_2$) such that
\begin{equation}
\mathbb{E}_{\varepsilon \sim \nu}\left[\| f_{\theta_0}^\tau(Au + \epsilon) - Wu \|_{\ell^2}^2\right] \leq c^{*} \delta^2.
\label{eq:fine}    
\end{equation}
\label{lem:fine}
\end{lemma}
\begin{proof} We start by considering that
\[
\begin{aligned}
&\mathbb{E}_{\varepsilon \sim \nu}\left[\| f_{\theta_0}^\tau(Au + \epsilon) - Wu \|_{\ell^2}^2\right] = \int_{Y_q} \| f_{\theta_0}^\tau(Au + \epsilon) - Wu \|_{\ell^2}^2 d_\nu(\epsilon) \\
& \quad = \int_{\| \varepsilon \| < \delta } \| f_{\theta_0}^\tau(Au + \epsilon) - Wu \|_{\ell^2}^2 d_\nu(\epsilon)  + \int_{\| \varepsilon \| > \delta } \| f_{\theta_0}^\tau(Au + \epsilon) - Wu \|_{\ell^2}^2 d_\nu(\epsilon).
\end{aligned}
\]
We now employ \eqref{eq:superconv} on the first term and H\"older inequality on the second term. Moreover, in view of Chebyshev's inequality, $\nu(\{ \| \varepsilon \| > \delta \}) \leq \frac{\sigma^2}{\delta^2}$. Therefore,
\[
\begin{aligned}
\mathbb{E}_{\varepsilon \sim \nu}\! \! \left[\| f_{\theta_0}^\tau(Au + \epsilon) - Wu \|_{\ell^2}^2\right] &\leq c_5^2 \delta^2 \left( 1- \frac{\sigma^2}{\delta^2}\right) + \frac{\sigma}{\delta} \left( \mathbb{E}_{\varepsilon \sim \nu}\! \left[\| f_{\theta_0}^\tau(Au + \epsilon) - Wu \|_{\ell^2}^4\right] \right)^{\frac{1}{2}} \\
& \leq c_5^2 \delta^2  + \frac{\sigma}{\delta} \left( 8 k_1^4 + 8 k_2^4 \mathbb{E}\big[\|\epsilon\|_{Y_q}^4\big] \right)^{\frac{1}{2}}.
\end{aligned}
\]
By \eqref{eq:moments} and by $\sigma = \delta^\eta$ we immediately verify the first thesis, and imposing $\eta = 3$ and $\delta^{2\eta}q < 1$ we get \eqref{eq:fine} with $c^{*} = c_5^2 + 2\sqrt{2} k_1^2 + 2 \sqrt{6} k_2^2$.
\end{proof}
In view of \cref{lem:fine}, we can easily prove the following convergence result regarding the the optimal network $f^\tau_{\theta^*}$:
\begin{proposition}
   Consider $\varepsilon \sim N(0,\sigma^2 I_q)$ with $\delta = \sigma^{1/3}$ and let $\theta^*$ satisfy \eqref{eq:thetastar}. There exists $\sigma_1 > 0$ such that, for $\sigma \leq \min\{\sigma_1, q^{1/2}\}$, it holds
   \begin{equation}
   \mathcal{L}(\theta^*; \mu, \nu) \leq c^{*} \delta^2.
       \label{eq:convthetastar}
   \end{equation}
\label{prop:convthetastar}
This also accounts to say that the random variable $f^\tau_{\theta^*}(A_\vN u + \epsilon)$ converges to $Wu$ in mean as $\delta \rightarrow 0$.
\end{proposition}
\begin{proof}
By definition of $\theta^*$ and by \cref{lem:fine},
\[
\begin{aligned}
\mathcal{L}(\theta^*; \mu, \nu) &\leq \mathcal{L}(\theta_0; \mu, \nu) = \int_{\mathcal{B}} \int_{Y_q} \| f_{\theta_0}^\tau(A_\vN  u + \epsilon) - Wu \|^2_{\ell^2} d_\nu(\epsilon) d_\mu(u)  \\
& \leq \int_{\mathcal{B}} c^{*} \delta^2 d_\mu(u) = c^{*} \delta^2.
\end{aligned}
\]
\end{proof}
\par
Although the optimal network $f^\tau_{\theta^*}$ allows for a precise approximation of the solution map of the inverse problems, it is impractical to solve the minimization problem stated in \eqref{eq:thetastar}. Instead, Neural Network algorithms require to draw a sample from the random variables $U$ and $E$ and to find the parameter $\theta$ which allows for the best reconstruction on such sample (training process). 
This task is addressed by minimizing a discretized loss functional, as reported in \cref{supervised_learning}, and results in the definition of the \textit{trained} Neural Network. The quality of the trained network can be verified by analyzing its \textit{generalization}, namely, its ability to provide good predictions even when tested on data outside the training sample. 
Such an analysis has been performed in detail (although with some different assumptions with respect to the ones in this work) in \cite{de2019deep}, and can be extended also to the problem in consideration.


\section{In practice: the particular case of CT}
\label{sec:in_practice} 

In this section, we focus on the practical aspects of the reconstruction algorithm introduced in \cref{subsec:cnn_converg}, in the particular case of limited-angle computed tomography (LA-CT) with the discrete setting. 
In the remainder of the article, the discrete counterpart of the operator  $A_\vN$  representing the LA-CT will be denoted by $R_\Gamma$.
We first define the regularized minimization problem, and then propose an effective method for the computation of the convolutional kernel filters approximating the backprojection operator in the wavelet domain. Thirdly, we present and discuss the general reconstruction workflow and finally give more details on the two CNN architectures we propose in this paper.

\subsection{The CT minimization problem}
\label{subsec:min_prob_ct}
After suitable discretization, we are given the finite-dimensional measurement vector:
\begin{equation}
    \mathbf{m} = \mathbf{R_\Gamma u^\dagger} + \mathbf{\epsilon},
\end{equation}
where $\mathbf{u^\dagger}\in \mathbb{R}^{p}$ denotes the (unknown) discrete and vectorized image, $\mathbf{R_\Gamma} \in \mathbb{R}^{q\times {p}}$ describes a discretized version of the Radon transform where the angles are limited in the arc specified by  $\Gamma$
and $\mathbf{\epsilon}\in \mathbb{R}^q$ models the measurement noise. 
We call $\mathbf{w}^\dagger$ the $\mathbb{R}^{p}$-vector such that $\mathbf{Wu^\dagger=w^\dagger}$ where $\mathbf{W} \in \mathbb{R}^{{p}\times {p}}$ represents a discretization of the wavelet transform. 
Thus, the regularized minimization problem is given by:
\begin{equation}
    \min_{\mathbf{w}\in\mathbb{R}^{p}} \|\mathbf{R_\Gamma W^\ast w-m}\|_2^2 + \lambda \|\mathbf{w}\|_1 
\end{equation}

Our recovery algorithm for finding a reconstruction $\mathbf{u}$ of $\mathbf{u}^\dagger$ involves convolutional architectures incorporated in the iterative structure of standard ISTA, as described in \cref{sec:unrolled_ista_theo}. In the next paragraphs, we detail the implementation of such an algorithm.

\subsection{Convolutional kernel operator for limited-angle CT}
\label{subsec:creation_bowties}

Building a convolutional algorithm that reproduces the behavior of standard ISTA first requires to identify the various blocks of the matrix $\mathbf{K}$ representing the backprojection operator in the wavelet domain. In other words, the very first step in the development of our method is to establish the convolutional filters of $\mathbf{K}$
which, once applied as defined in \cref{eq:madness}, provide a reliable approximation of the operator $\mathbf{W R_\Gamma^\ast R_\Gamma W}^\ast$. 

One way to compute such convolutional filters that proves to be a numerically advantageous alternative to \cref{eq:madness2}, is represented in \cref{fig:bowtie_computation}. 
Let us consider an object whose representation in the wavelet
domain has only one nonzero pixel, located at the center of one of its wavelet subbands. Applying the operator $\mathbf{W R_\Gamma^\ast R_\Gamma W}^\ast$ to this initial object leads to a new object whose subbands present a bowtie-shaped structure. Those 'bowtie' subbands constitute a first set of convolutional filters. By reiterating this operation until the central pixel of all the wavelet subbands in the initial object has been visited, one obtains the entire collection of convolutional filters necessary for the approximation of $\mathbf{W R_\Gamma^\ast R_\Gamma W}^\ast$.

\begin{figure}
    \centering
    \includegraphics[width=\textwidth]{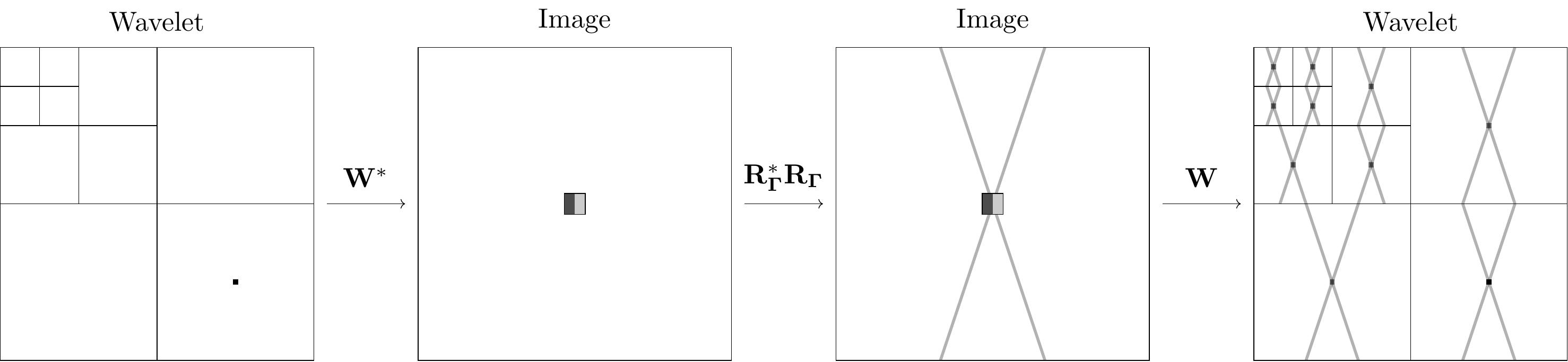}
    \caption{Illustration of the proposed way to compute the filters of the convolutional kernel operator $K$ in the limited-angle CT case. The initial object (on the left) is created such that all its pixels but one are set to zero. The only non-zero pixel is located at the center of one of its wavelet subbands.  By applying the operator $\mathbf{W R_\Gamma^\ast R_\Gamma W}^\ast$ to this initial object, one obtains a new object in the wavelet domain, whose subbands present a bowtie-shaped structure. The set of 'bowtie' subbands thus computed from all the possible initial objects constitute the filters of the convolutional kernel operator.
    Here we have represented three levels of decomposition in the wavelet domain, meaning that the total number of convolutional filters amounts to $(3^2+1)^2 = 10^2 = 100$.}
    \label{fig:bowtie_computation}
\end{figure}

\begin{figure}
    \centering
    \includegraphics[width=\textwidth]{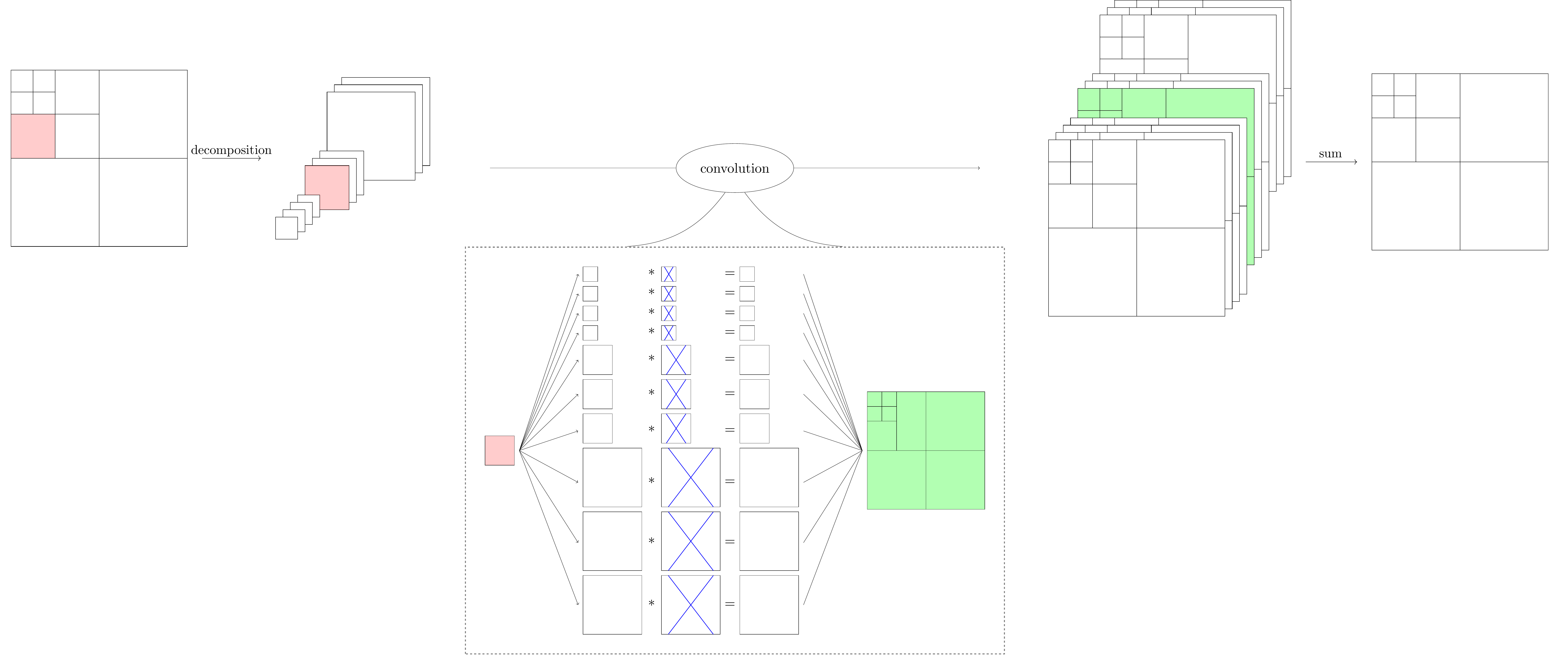}
    \caption{Illustration of the way the filters of the convolutional kernel operator are applied to each wavelet subband of the initial object, after up- and down- sampling operations, in order to approximate the operator $\mathbf{W R_\Gamma^\ast R_\Gamma W}^\ast$.}
    \label{fig:bowtieNET}
\end{figure}

In order to imitate the behavior of the operator $\mathbf{W R_\Gamma^\ast R_\Gamma W}^\ast$, the convolutional filters so computed are then to be applied to the wavelet subbands of an object as illustrated in \cref{fig:bowtieNET}. First, each wavelet subband of the object of interest is replicated $3(J-J_0)+1$ times. Those replicas are either upsampled, or downsampled, or kept with the same dimensions, depending on the scale of the filter they are to be convolved with. 
The set of convolutional filters used on the replicas of a particular wavelet subband of scale $j$ and type $(t)$ is the set of filters beforehand generated by applying the operator $\mathbf{W R_\Gamma^\ast R_\Gamma W}^\ast$ to an object whose only non-zero pixel is located at the center of its wavelet subband of exact same scale $j$ and type $(t)$. Once the convolutions between the replicas and the filters have been performed, the resulting subbands are reassembled to form the wavelet representation of a new object. This process is reiterated for all the subbands of the initial object and ultimately, the $3(J-J_0)+1$ resulting items are summed. The final outcome is an approximation of the wavelet representation of the operator $\mathbf{W R_\Gamma^\ast R_\Gamma W}^\ast$ applied to the initial object (cf \cref{fig:circle_backproj}). 

Two remarks are worth mentioning regarding the creation and use of the above-defined convolutional filters. First, our practical implementation very slightly differs from the theory presented in \cref{eq:madness} as far as the downsampling is concerned. In our codes, downsampling is indeed applied before computing the convolution between the filter and the wavelet subband replica, and not after as it is presented in the theory. This choice is motivated by the reduction in terms of storage needs and running time such a change allows while preserving the accuracy of the approximation. Secondly, both the theoretical analysis and the experimental tests showed that the dimensions of the convolutional filters used for the approximation of $\mathbf{W R_\Gamma^\ast R_\Gamma W}^\ast$ do affect the accuracy of the results. Initially, we assumed that the convolutional filters should be generated from only-one-non-zero-pixel objects with the same dimensions $2^J\times 2^J$ as the image of interest (recall that $p=2^{2J}$). However, we reached the conclusion that they actually have to be generated from twice bigger objects, that is of dimensions $2^{J+1}\times 2^{J+1}$, in order to get an accurate approximation of the operator $\mathbf{W R_\Gamma^\ast R_\Gamma W}^\ast$. An illustration of the effects of the size of the convolutional filters can be seen on \cref{fig:circle_backproj}.

\begin{figure}[t]
  \centering
  \begin{subfigure}[b]{0.3\textwidth}
     \centering \includegraphics[width=\textwidth,trim={3cm 1cm 3cm 1cm},clip]{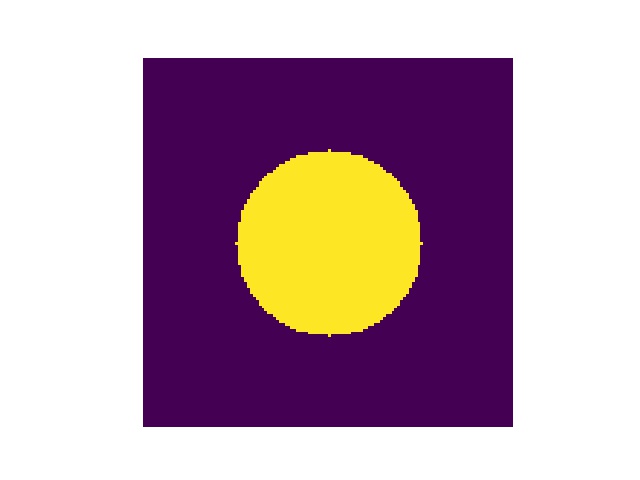}
     \caption{}
     \label{fig:ground_truth}
  \end{subfigure}
  \begin{subfigure}[b]{0.3\textwidth}
     \centering \includegraphics[width=\textwidth,trim={3cm 1cm 3cm 1cm},clip]{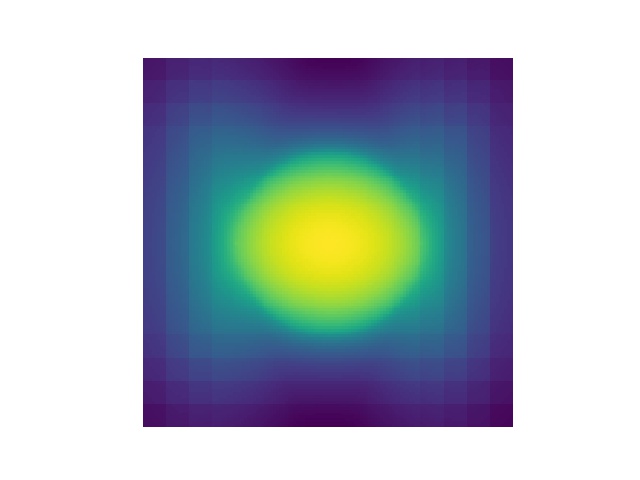}
     \caption{}
     \label{fig:bp_hat_small}
  \end{subfigure}
  \begin{subfigure}[b]{0.3\textwidth}
     \centering \includegraphics[width=\textwidth,trim={3cm 1cm 3cm 1.5cm},clip]{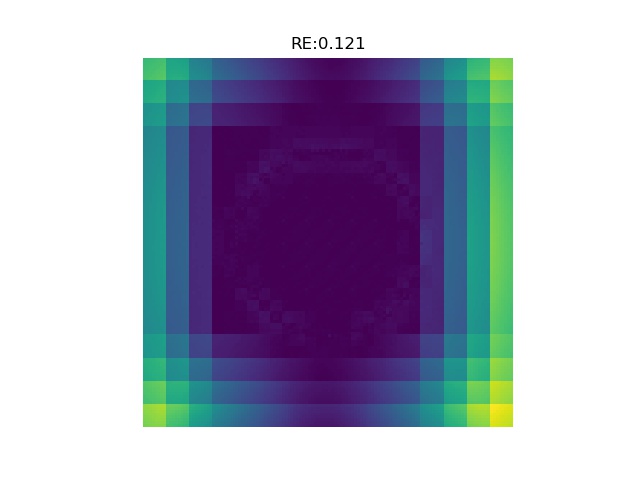}
     \caption{RE: 0.121}
     \label{fig:bp_hat_small_err}
  \end{subfigure}
  
  \begin{subfigure}[b]{0.3\textwidth}
     \centering \includegraphics[width=\textwidth,trim={3cm 1cm 3cm 1cm},clip]{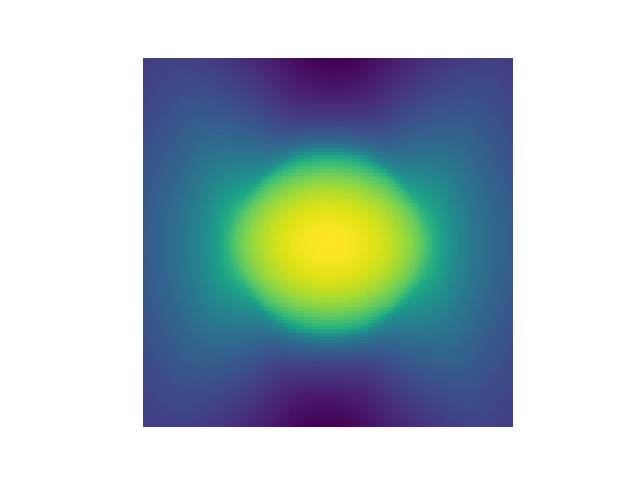}
     \caption{}
     \label{fig:bp_true}
  \end{subfigure}
  \begin{subfigure}[b]{0.3\textwidth}
     \centering \includegraphics[width=\textwidth,trim={3cm 1cm 3cm 1cm},clip]{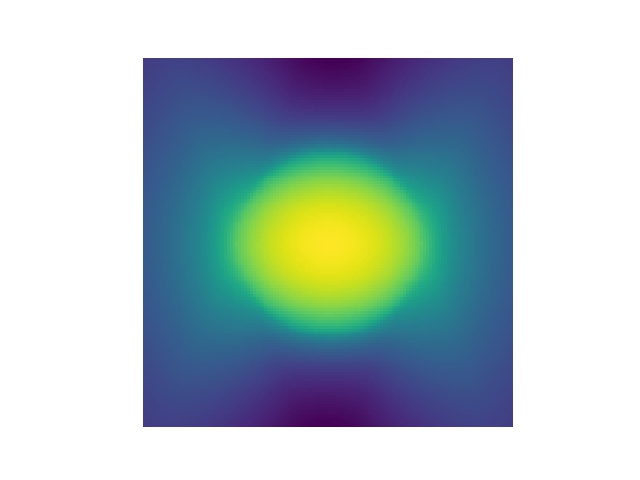}
     \caption{}
     \label{fig:bp_hat_big}
  \end{subfigure}
  \begin{subfigure}[b]{0.3\textwidth}
     \centering \includegraphics[width=\textwidth,trim={3cm 1cm 3cm 1.5cm},clip]{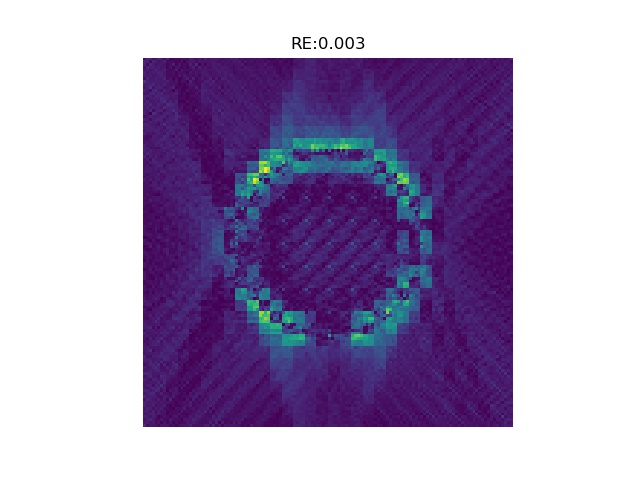}
     \caption{RE: 0.003}
     \label{fig:bp_hat_big_err}
  \end{subfigure}
  \caption{Illustration of the effect of the standard backprojection operator and of its approximations based on convolutional filters of different sizes. (\subref{fig:ground_truth}) shows the ground truth $\mathbf{u}^\dagger$ of interest and (\subref{fig:bp_true}) its standard backprojection $\mathbf{R_\Gamma^\star R_\Gamma u}^\dagger$, computed with the basic functions of the Python package \texttt{scikit-image}. (\subref{fig:bp_hat_small}) (resp. (\subref{fig:bp_hat_big})) represents the approximation $\mathbf{Ku}^\dagger$ obtained with convolutional filters beforehand generated from $2^J\times 2^J$ (resp. $2^{J+1}\times 2^{J+1}$) only-one-non-zero-pixel object.  (\subref{fig:bp_hat_small_err}) and  (\subref{fig:bp_hat_big_err}) show the absolute differences between the approximation of the backprojection operator and the expected value (\subref{fig:bp_true}). The dynamic range of the plot is modified for better contrast.}
  \label{fig:circle_backproj}
\end{figure}

\subsection{Our CNN architectures}
\label{sec:our_models}
The above described method for generating and applying the filters of the  kernel operator $\mathbf{K}$ makes the concrete implementation of a convolutional algorithm that imitates the behavior of standard ISTA possible. 
Thus, the convolutional implementation of ISTA, result-wise equivalent to the standard one, could be written as:
\begin{equation}
    \mathbf{w}^{(n+1)} = \mathcal{S}_{\frac{\lambda}{L}} \left(\mathbf{w}^{(n)} + \frac{1}{L}\left(\mathbf{WR_\Gamma^\ast m} - \mathbf{K w}^{(n)}\right)\right), \quad \quad n = \{0,\dots,N\}
\end{equation}

This algorithm offers the merits of the iterative model-based method ISTA, while allowing the incorporation of data-driven approaches, such as machine learning and deep neural network techniques. The implementation of 
$\mathbf{K}$ indeed involves operations that are all perfectly adaptable to the framework of CNNs. 
Our goal is precisely to take full advantage of this compatibility and profit from the remarkable potentials of deep neural networks by converting the hitherto fixed operator $\mathbf{K}$ into a partially trainable CNN. 
Thus, the center of the convolutional filters so far precomputed with the deterministic method presented in \cref{subsec:creation_bowties} can henceforth be considered as parameters to be learned from data. 
The choice of learning only the central part of the convolutional filters of $\mathbf{K}$ rather than the whole filters is motivated by the need to reduce the model complexity 
which, in the latter case, makes the training of the model burdensome if not impractical.

In order to further improve reconstruction performance, we also propose to learn the soft-thresholding parameter as well as the step-length so far set at $1/L$.
The so-defined convolutional architecture results in our proposed algorithm $\Psi$DONet, whose convergence results are detailed in \cref{subsec:cnn_converg}. 
In \cref{subsubsec:model_bowtie_holes} and \cref{subsubsec:model_astra_NN}, we propose two different implementations of $\Psi$DONet, that proves to be result-wise similar as it can be observed in \cref{subsec:results}.

\subsubsection{$\Psi$DONet-F}
\label{subsubsec:model_bowtie_holes}
The most natural way to implement $\Psi$DONet consists in partitioning the convolutional operator $\mathbf{K}$ into two operators: a fixed one, $\breve{\mathbf{K}}^\tau$, and a trainable (single-layer) CNN referred to as $\layer_{\zeta}^\tau$, where $\tau$ is a tunable hyperparameter. The two operators have the exact same architecture as $\mathbf{K}$  and their sum, before any training, is strictly equivalent to $\mathbf{K}$. The first operator $\breve{\mathbf{K}}^\tau$ is non trainable and its filters are a copy of the filters of $\mathbf{K}$ with the exception that the $\tau \times \tau$ central weights of each filter are set to zero. The second operator $\layer_{\zeta}^\tau$, on the contrary, is composed by $\tau\times\tau$-trainable filters that are initialized with the $\tau\times\tau$ central part of the filters of  $\mathbf{K}$ (cf \cref{fig:partition_k_holes}). 
This first implementation of $\Psi$DONet, referred to as Filter-Based $\Psi$DONet or $\Psi$DONet-F, is formulated as:

\begin{equation}
    \mathbf{w}^{(n+1)} = \mathcal{S}_{\gamma_n} \left(\mathbf{w}^{(n)} + \alpha_n\left(\mathbf{WR_\Gamma^\ast m} - \beta_n \left(  \breve{\mathbf{K}}^\tau\mathbf{w}^{(n)} +  \layer_{\zeta_n}^\tau\mathbf{w}^{(n)}\right)\right)\right)
\label{eq:cnn_model_bowtie}
\end{equation}
where $n = \{0,\dots,N\}$ and the parameters to be learned are $\{\gamma_0,\alpha_0, \beta_0,\zeta_0$,$\dots$,$\gamma_N,\alpha_N$, $\beta_N,\zeta_N\}$. The parameters $\{\beta_0,\dots,\beta_N\}$ have been added in such a way that the influence of the fixed operator $\breve{\mathbf{K}}^\tau$ with respect to the constant term $\mathbf{WR_\Gamma^\ast m}$ can be adjusted in order to maximize the accuracy of the results. It is worth mentioning that for the particular choice of $\gamma_n=\frac{\lambda}{L}, \alpha_n=\frac{1}{L}, \beta_n=1, $ for any $ n = \{0,\dots,N\} $, this model before any training is exactly equivalent to standard ISTA.

The trade-off between the number of parameters that can be improved through the learning process and the trainability of the model is controlled by $\tau$. For a sound choice of such a hyperparameter, the complexity of the model is sufficiently reduced to allow for the convergence of the learning algorithm  while enabling the enhancement of a significant number of weights in the filters. 

\begin{figure}
    \centering
    \includegraphics[width=\textwidth]{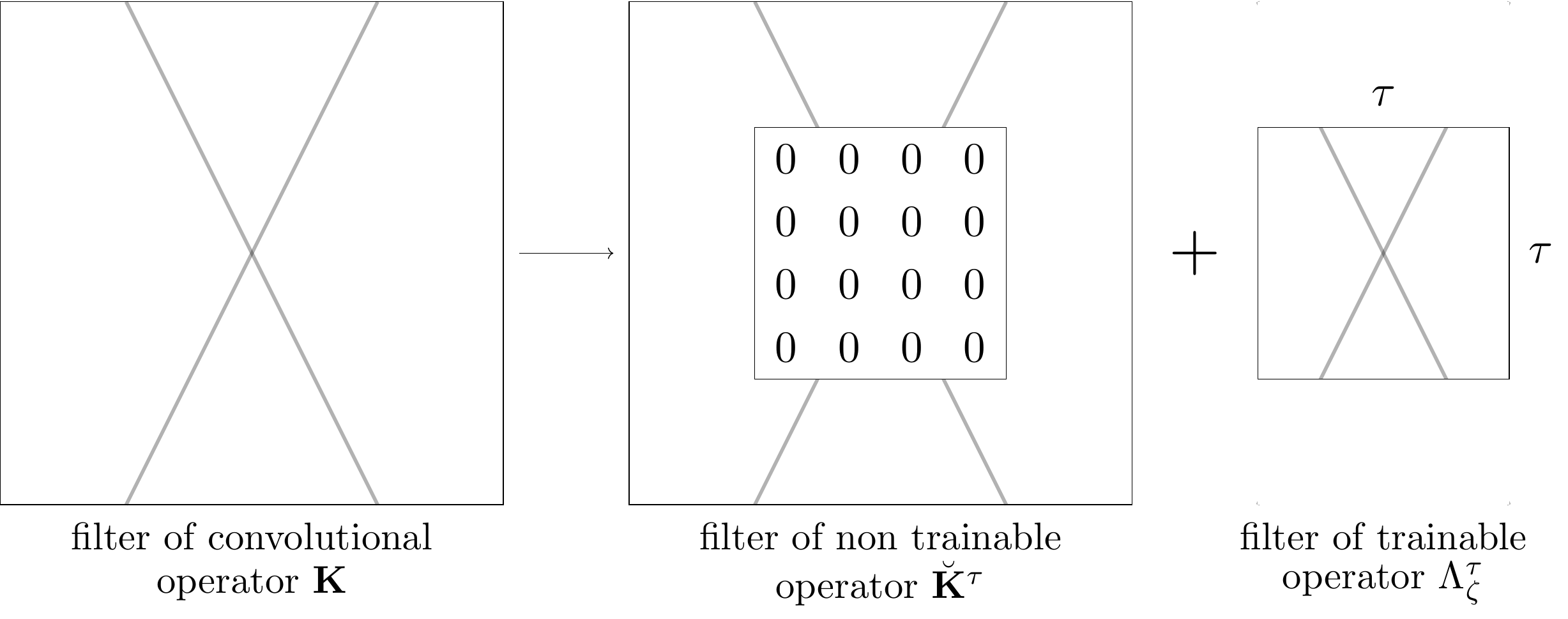}
    \caption{Illustration of the way the convolutional filters of the two operators in $\Psi$DONet-O are computed based on the filters of operator $\mathbf{K}$. Each filter of $\mathbf{K}$ is partitioned into two filters, which sum is equivalent to the initial one. The filter of $\breve{\mathbf{K}}^\tau$ is a copy of the filter of $\mathbf{K}$ with the exception that the $\tau \times \tau$ central weights are set to zero. The filter of $\layer_{\zeta}^\tau$ has dimensions $\tau\times\tau$ and is initialized with the central $\tau \times \tau$ central weights of the filter of $\mathbf{K}$.}
    \label{fig:partition_k_holes}
\end{figure}

This implementation has the advantage of offering a clear interpretation of the role and meaning of the convolutional filters belonging to $\breve{\mathbf{K}}^\tau$ and $\layer_{\zeta}^\tau$. Those filters are indeed initialized with the filters of the operator $\mathbf{K}$ that imitates the behavior of $\mathbf{W R_\Gamma^\ast R_\Gamma W}^\ast$. Thus, modifying their weights through the learning process can be thought of as a direct improvement of the back-projection operator. 

$\Psi$DONet-F has led to very satisfactory preliminary results, presented in \cref{sec:experiments}. However, training such a model on big images may quickly become extremely onerous in terms of running time and storage requirements. Such problems may arise while training $\Psi$DONet-F on images of dimensions greater or equal to $256\times 256$.
Unlike typical CNNs that usually make use of small-sized convolutional filters, the filters of $\breve{\mathbf{K}}^\tau$ in our proposed algorithm are much bigger than the wavelet subbands they are convolved with. This uncommon procedure, that \textit{inter alia} implies the padding, \textit{i.e.}, the addition of many extra pixels to the edge of each wavelet subband, brings about a severe speed reduction in the training process as well as the necessity of a substantial memory space. The alternative implementation of $\Psi$DONet, described in \cref{subsubsec:model_astra_NN} , addresses these shortcomings.

\subsubsection{$\Psi$DONet-O}
\label{subsubsec:model_astra_NN}
The main flaw of $\Psi$DONet-F rests upon the use of operator $\breve{\mathbf{K}}^\tau$ which implies numerous burdensome convolutions.
This issue is worked around in $\Psi$DONet-O \eqref{eq:model_astra_NN}, as $\breve{\mathbf{K}}^\tau$ is not involved anymore. Here, the backprojection operator is not approximated, meaning that $\mathbf{W R_\Gamma^\ast R_\Gamma W}^\ast$ is indeed implemented as the succession of the inverse wavelet, Radon, inverse Radon and direct wavelet transforms applied to the iterate $\mathbf{w}^{(n)}$. This second implementation of $\Psi$DONet, named Operator-Based $\Psi$DONet or $\Psi$DONet-O, reads as:

\begin{equation}
    \mathbf{w}^{(n+1)} = \mathcal{S}_{\gamma_n} \left(\mathbf{w}^{(n)} + \alpha_n\left(\mathbf{WR_\Gamma^\ast m} -  \mathbf{W R_\Gamma^\ast R_\Gamma W}^\ast\mathbf{w}^{(n)}\right) +  \beta_n \layer_{\zeta_n}^\tau\mathbf{w}^{(n)}\right)
\label{eq:model_astra_NN}
\end{equation}
where $n = \{0,\dots,N\}$, the parameters to be learned are $\{\gamma_0,\alpha_0, \beta_0,\zeta_0,\dots,\gamma_N $,$\alpha_N $, $\beta_N,\zeta_N\}$, and $\layer_{\zeta_n}^\tau$
has the same architecture as the operator as $\mathbf{K}$. The block diagram of the method is represented in \cref{fig:diagram_unrolled}.
For the special choice of $\gamma_n=\frac{\lambda}{L}, \alpha_n=\frac{1}{L}, \beta_n=0, $ for any $ n = \{0,\dots,N\} $, this model is exactly equivalent to standard ISTA. The only convolutions involved in this alternative implementation are the ones composing the CNN $\layer_{\zeta_n}^\tau$, whose filters are chosen to be small enough to avoid any running time or storage issue. 
In that sense, $\Psi$DONet-O offers an implementation 
numerically preferable to $\Psi$DONet-F, while 
retaining the same properties on a theoretical level.
Furthermore, such a model keeps offering a clear interpretation of its post-processing abilities since $\layer_{\zeta_n}^\tau$, on account of its architecture, can still be seen as an adjunct for improving the back-projection operator. 

\begin{figure}
    \centering
    \includegraphics[width=\textwidth]{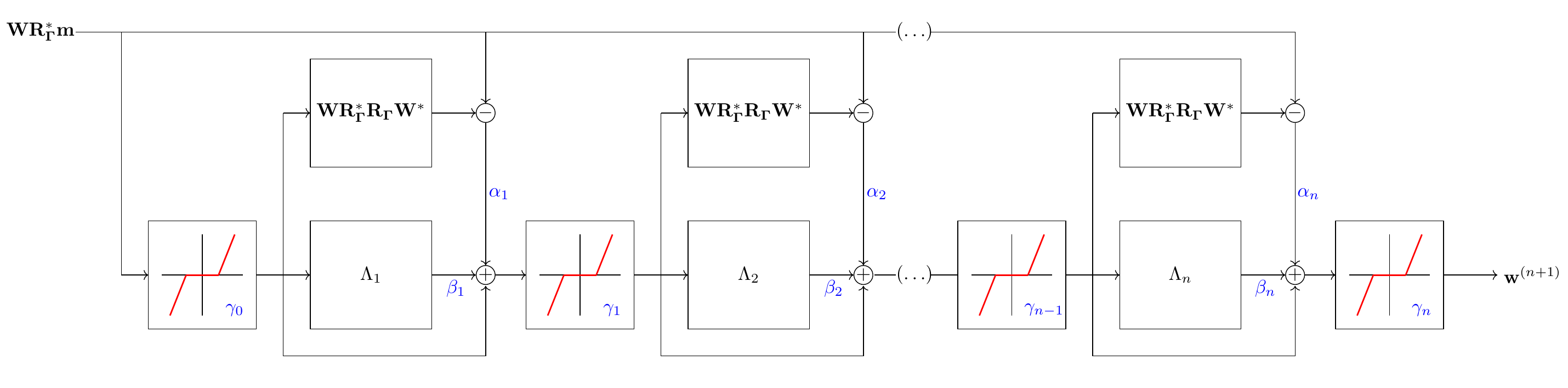}
    \caption{Block diagram of the proposed model \eqref{eq:model_astra_NN}}
    \label{fig:diagram_unrolled}
\end{figure}

\subsubsection{Note on soft-thresholding parameters} From a theoretical point of view, the parameters $\gamma_0, \dots, \gamma_N$ in $\Psi$DONet-F and $\Psi$DONet-O have to be non-negative, as they represent the soft-thresholding parameters. In order to stick to 
the operator originally involved in standard ISTA, it is possible to enforce the positivity of the coefficient by replacing each $\gamma_n$ by $10^{\tilde{\gamma}_n}$, where $\tilde{\gamma}_n$ becomes the actual trainable parameter. 
However, in order to allow for a greater degree of freedom in the learning process, we decided to implement the operator $\mathcal{S}_{\gamma_n}$ in such a way that it is also interpretable for negative values of its parameter $\gamma_n$. In such a case,
we define the operator $\mathcal{S}_{\gamma_n<0}$ as the symmetric of the soft-thresholding curve with respect to $y=x$, while
for non-negative values of $\gamma_n$, $\mathcal{S}_{\gamma_n}$ is exactly equivalent to the soft-thresholding operator.
Formally, $\mathcal{S}_{\gamma_n}$ becomes:

\vspace{.25cm}
\minibox[c]{  For $\gamma_n\geq 0:$ \\[0.25em]
$\mathcal{S}_{\gamma_n}(x) =   \begin{cases} x - \gamma_n, & \mbox{if } x \geq \gamma_n  \\ 0, & \mbox{if } |x|<\gamma_n 
\\ x + \gamma_n, & \mbox{if } x \leq - \gamma_n \end{cases}   $ 
} %
\hfill{\minibox[c]{
For $\gamma_n<0:$ \\[0.25em]
$\mathcal{S}_{\gamma_n}(x) =   \begin{cases} x - \gamma_n, & \mbox{if } x \geq 0
\\ x + \gamma_n, & \mbox{if } x < 0  \end{cases} $
}}
\vspace{.5cm}

The two implementations $\Psi$DONet-F and $\Psi$DONet-O are tested with and without the positivity constraint on $\gamma$ (cf results in \cref{subsec:results}).

\subsection{Supervised Learning}
\label{supervised_learning}
If we denominate $f_\mathbf{\theta}^\tau$ the $N$-layer CNN that, given $\mathbf{m}$, computes the final output $\mathbf{w}_{N+1}$ according to one of the two proposed architectures, we aim at learning the optimal high-dimensional vector $\mathbf{\theta} = \{\gamma_0,\alpha_0, \beta_0,\zeta_0,\dots,$ $\gamma_N,\alpha_N,\beta_N,\zeta_N\}$ that ideally satisfies the relation:
\begin{equation}
    f_\mathbf{\theta}^\tau(\mathbf{m}) \approx \mathbf{Wu}^\dagger 
\end{equation}

For a mathematical formalization, we regard the tuple $\left( \mathbf{m}, \mathbf{u}^\dagger\right) \in \mathbb{R}^{q}\times\mathbb{R}^{p}$ as a random variable with a joint probability distribution $\Xi$, as detailed in \cref{subsec:cnn_converg}. Ideally, we would like to find a parameter vector $\mathbf{\theta}^*$ minimizing the expected risk:
\begin{equation}
    \min_\mathbf{\theta} \left(\mathbb{E}_{( \mathbf{m}, \mathbf{u}^\dagger) \sim \Xi} \| f_\mathbf{\theta}^\tau(\mathbf{m}) - \mathbf{Wu}^\dagger \|_2^2 \right)
    \label{eq:min_pb_contiunous}
\end{equation}
Other loss functions, such as the weighted $l_2$-norm, where the wavelet coefficients are weighted depending on their scale, have been tested and lead to results similar to the non-weighted $l_2$-norm. For the sake of brevity, we will stick to the basic form of \eqref{eq:min_pb_contiunous}.

In practice, computing the expectation with respect to $\Xi$ is not possible. Instead, we are typically given a finite set of independent drawings $( \mathbf{m}_1,\mathbf{u}_1^\dagger ),\dots,( \mathbf{m}_S,\mathbf{u}_S^\dagger)$ and consider the minimization of the empirical risk:
\begin{equation}
    \min_\mathbf{\theta} \frac{1}{S} \sum_{i=1}^S \| f_\mathbf{\theta}^\tau(\mathbf{m}_i) - \mathbf{Wu}_i^\dagger \|_2^2 
    \label{eq:min_pb_empirical}
\end{equation}
Depending on the properties of $f_\mathbf{\theta}^\tau$, the optimization problem is in general non-convex. In the case of neural networks, typically some form of gradient descent is used, where the gradients are calculated via backpropagation \cite{rumelhart1986}. Computing the gradient over the entire training set in \eqref{eq:min_pb_empirical} is often not feasible for large-scale problems due to memory limitations. To circumvent this problem, stochastic of minibatch gradient descent is used, in which the gradient is approximated over smaller, randomly selected batches of training examples \cite[Chapter~8]{Goodfellow-et-al-2016}.

The final performance (\textit{i.e.}, the generalization) of the trained map $f_\mathbf{\theta}^\tau$ is evaluated on a separate set of independent drawings, the so-called test set, that were not previously used in the optimization of $\mathbf{\theta}$ in \eqref{eq:min_pb_empirical}.

\section{Experiments and results}
\label{sec:experiments}
In this section, we evaluate the performance of the proposed reconstruction schemes by comparing with standard ISTA. 

\subsection{Preliminaries}
\label{subsec:preliminaries}
Let us begin by describing the considered experimental scenario, the implementation of the used operators and the training procedure.

\subsubsection{Data set}
\label{subsubsec:data_set}
The data set consists of 10700 synthetic images of ellipses, where the number, locations, sizes and the intensity gradients of the ellipses are chosen at random. Using the Matlab function \texttt{radon}, we simulate measurements for a missing wedge of $60^\circ$ with Gaussian noise. To avoid an inverse crime \cite{siltanen_mller2012} the measurements are simulated at a higher resolution and then downsampled for an image resolution of $128 \times 128$. 10000 images are used for training, 200 images for validation and 500 for testing. 

\subsubsection{Operators}
\label{subsubsec:operators}
For the implementation of the discrete limited angle operator $\mathbf{R}_\Gamma$ we use the \texttt{radon} routine of the Python package \texttt{scikit-image} \cite{scikit-image}, or the 2D parallel beam geometry of the Operator Discretization Library (ODL) library \cite{odl-toolbox}, which is based on the Astra toolbox \cite{astra2016}. The former is employed for generating the backprojections $WR_\Gamma^\ast m$ provided as inputs to $\Psi$DONet-F and $\Psi$DONet-O, while the latter is used for the implementation of $\mathbf{W R_\Gamma^\ast R_\Gamma W}^\ast$ in $\Psi$DONet-O. The direct and inverse Radon transform operators are multiplied by a constant so that their norm is equal to one.
Regarding the wavelet transform, we make use of the Python package \texttt{pywt} \cite{pywt2019} or a rectified version of the package \texttt{tf-wavelets} \cite{tf-wavelets}. 
In all our experiments, we consider the case $J = 7$ and $J_0=3$, implying that the wavelet decomposition $\mathbf{Wu}$ has $10$ subbands. For the two $\Psi$DONet-F and $\Psi$DONet-O, we choose to set $\tau$ to 32. Note that according to theory, $\tau$ is supposed to be odd, however, in practice we prefer it to be even. This very slight modification has no effect on the results.

\subsubsection{Network structure and training}
\label{subsubsec:network_training}
For the implementation of the two $\Psi$DONet-F and $\Psi$DONet-O, we fix the number of unrolled blocks $N$ to 120. 
In order to reduce the number of parameters to be learned, we choose to use only 40 different sets of trainable parameters $\{\zeta_n,\gamma_n, \alpha_n,\beta_n\}$, each of which being used over 3 consecutive blocks, instead of the theoretically expected 120 sets. 
Implementing and training our algorithms has been performed using Tensorflow with an Adam optimizer \cite{adam_2015} and a learning rate (step size) of $10^{-3}$. The number of epochs
was chosen to be 3, and the batch size was set to 25. 
The training, run on a NVIDIA Quadro P6000 GPU, roughly takes 20 hours.

\subsubsection{Compared Methods}
\label{subsubsec:compared_methods}
We compare the preliminary results of the architectures we propose with the reconstructions provided by standard ISTA. In the implementation of the latter, we make use of the formula introduced by Daubechies et al. in \cite{D}. The regularization parameter $\lambda$ and the constant $L$ are respectively set to $2.10^{-6}$ and $5$. The number of iterations for ISTA is determined by the stopping criterion: 
\begin{equation}
    \frac{\|\mathbf{u}^{(n+1)}-\mathbf{u}^{(n)}\|_2^2}{\|\mathbf{u}^{(n)}\|_2^2} < tol
\end{equation}
where $tol$ is chosen to be $2.10^{-4}$.
Hereinbelow, we give a list of the abbreviations henceforth used for the different recovery methods.

\begin{description}[leftmargin=!,align=right,labelwidth=\widthof{\bfseries hhhhhhhh}]
  \item[u$_{\text{ista}}$] Standard ISTA reconstruction. 
  \item[u$_{\text{FBP}}$] Standard filtered backprojection with the 'ramp' filter of \texttt{skicit-image}.  
 \item[u$_{\text{$\Psi$do-F}}^+$] Solution provided by $\Psi$DONet-F with positivity-constraint on the soft-thresholding parameter ($\gamma_n = 10^{\tilde{\gamma}_n}, \forall n$).
  \item[u$_{\text{$\Psi$do-F}}$] Solution provided by $\Psi$DONet-F without positivity-constraint on the soft-thresholding parameter.
    \item[u$_{\text{$\Psi$do-O}}^+$] Solution provided by $\Psi$DONet-O with positivity-constraint on the soft-thresholding parameter ($\gamma_n = 10^{\tilde{\gamma}_n}, \forall n$).
  \item[u$_{\text{$\Psi$do-O}}$]  Solution provided by $\Psi$DONet-O without positivity-constraint on the soft-thresholding parameter.

\end{description}

\subsubsection{Similarity Measures}
\label{subsubsec:similarity_measures}

For an assessment of image quality, we are using several quantitative measures, such as the relative error (RE) given by
\begin{equation}
    \|\mathbf{u}^\dagger-\mathbf{u}\|_2/  \|\mathbf{u}^\dagger\|_2
\end{equation}
where $u^\dagger$ denotes the reference image and $u$ its reconstruction. Furthermore, we consider the peak
signal-to-noise ratio (PSNR) and the structured similarity index (SSIM) \cite{wang2004} provided by Tensorflow. Finally,
we are reporting the Haar wavelet-based perceptual similarity index (HaarPSI) that was recently proposed
in \cite{haar_psi2018}.

\subsection{Results}
\label{subsec:results}
In the following, we will report and discuss the results of our numerical experiments.
The average image quality measures of the 500 test images are reported in \cref{tab:average_im_quality}. Furthermore, a visualization of the reconstruction
quality for two of the test images is given in \cref{fig:compar_methods_10016} and \cref{fig:compar_methods_10004}. Due to the large missing angle of $60^\circ$, the FBP
images in \cref{fig:im10016fbp} and \cref{fig:im10004fbp} are contaminated with streaking artifacts and contrast changes. The standard ISTA offers reconstructions of higher quality (cf \cref{fig:im10016ista} and \cref{fig:im10004ista}), however, the streaking artifacts are still noticeable as well as the impurities 
due to the noise in the measurements. 
Besides, the ISTA reconstructions are toned down, meaning that for the most part, the intensity of the pixels remain significantly lower than the expected values. 
\begin{figure}[H]
  \centering
  \captionsetup{justification=centering}
  \begin{subfigure}[b]{0.3\textwidth}
     \centering \includegraphics[width=\textwidth,trim={3cm 1cm 3cm 1cm},clip]{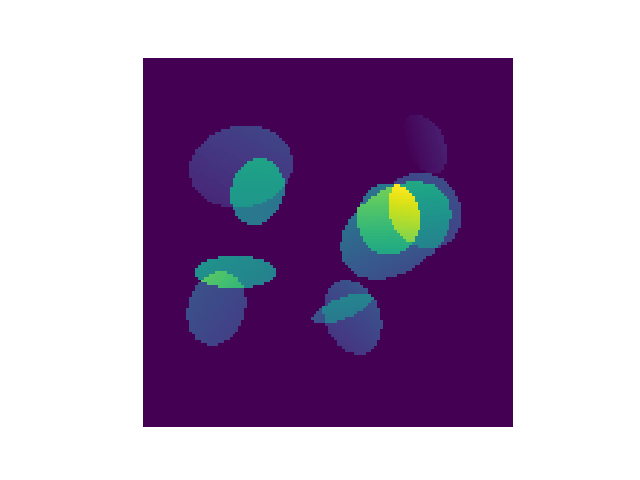}
     \caption{Ground truth $\mathbf{u}^\dagger$  \\{\color{white} RE: , SSIM: }}
     \label{fig:im10016groundtruth}
  \end{subfigure}
  \begin{subfigure}[b]{0.3\textwidth}
     \centering \includegraphics[width=\textwidth,trim={3cm 1cm 3cm 1cm},clip]{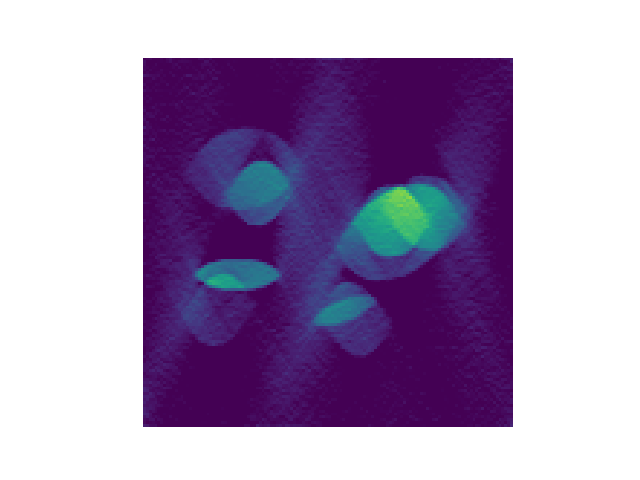}
     \caption{$\mathbf{u}_{\text{ista}}$ \\ RE: 0.47, SSIM: 0.24}
     \label{fig:im10016ista}
  \end{subfigure}
  \begin{subfigure}[b]{0.3\textwidth}
     \centering \includegraphics[width=\textwidth,trim={3cm 1cm 3cm .5cm},clip]{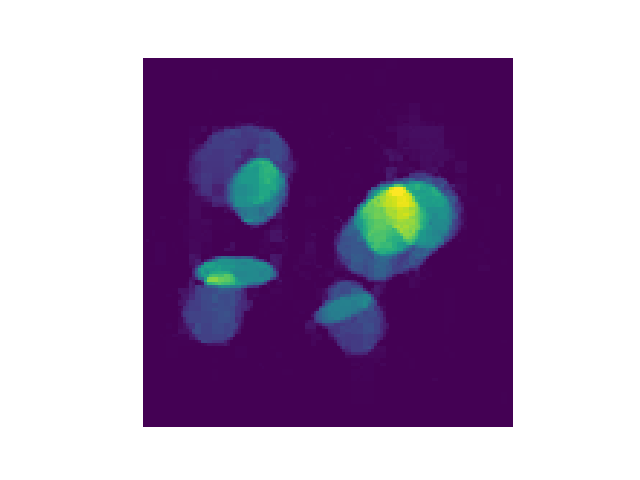}
     \caption{$\mathbf{u}_{\text{$\Psi$do-O}}$ \\ RE: 0.18 , SSIM: 0.83}
     \label{fig:im10016ourprediction_flacm1}
  \end{subfigure}
  
  \begin{subfigure}[b]{0.3\textwidth}
     \centering \includegraphics[width=\textwidth,trim={3cm 1cm 3cm 1cm},clip]{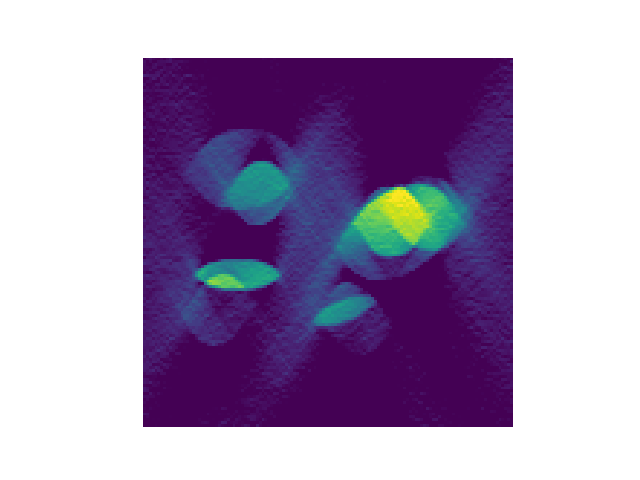}
     \caption{$\mathbf{u}_{\text{FBP}}$  \\ RE: 0.66, SSIM: 0.14}
     \label{fig:im10016fbp}
  \end{subfigure}
  \begin{subfigure}[b]{0.3\textwidth}
     \centering \includegraphics[width=\textwidth,trim={3cm 1cm 3cm 1cm},clip]{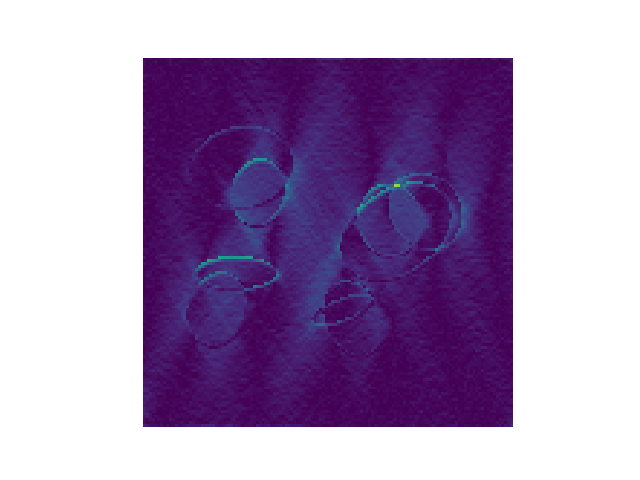}
     \caption{$|\mathbf{u}^\dagger-\mathbf{u}_{\text{ista}}|$ \\{\color{white} RE: , SSIM: } }
     \label{fig:im10016ista_err}
  \end{subfigure}
  \begin{subfigure}[b]{0.3\textwidth}
     \centering \includegraphics[width=\textwidth,trim={3cm 1cm 3cm .5cm},clip]{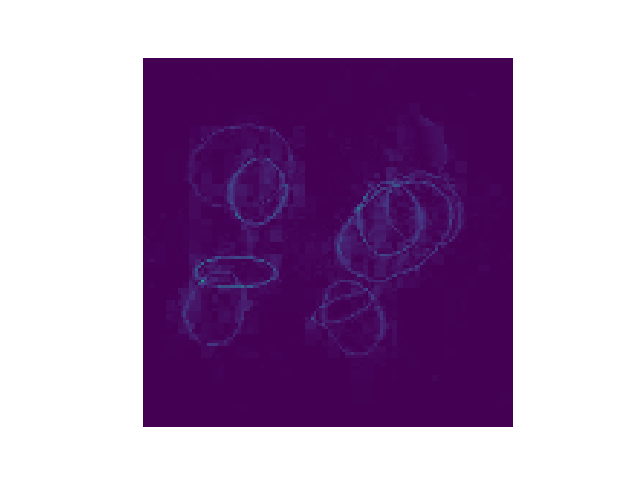}
     \caption{ $|\mathbf{u}^\dagger-\mathbf{u}_{\text{$\Psi$do-O}}|$ \\{\color{white} RE: , SSIM: }}
     \label{fig:im10016ourprediction_flacm1_err}
  \end{subfigure}
  
  \begin{subfigure}[b]{0.3\textwidth}
     \centering \includegraphics[width=\textwidth,trim={3cm 1cm 3cm 1cm},clip]{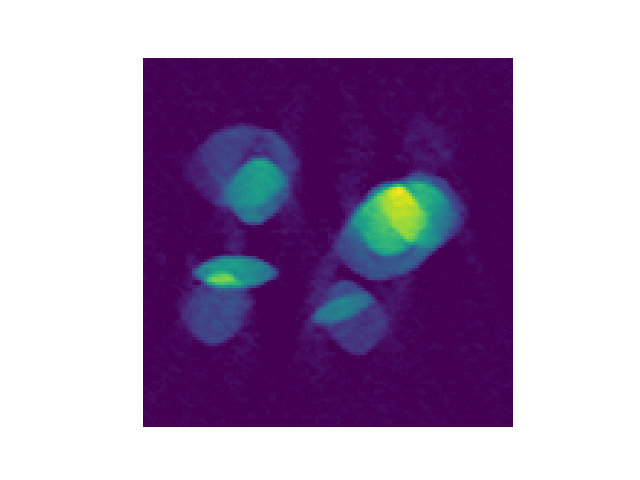}
     \caption{$\mathbf{u}_{\text{$\Psi$do-O}}^+$ \\ RE: 0.23, SSIM: 0.56}
     \label{fig:im10016ourprediction_flacm1exp}
  \end{subfigure}
  \begin{subfigure}[b]{0.3\textwidth}
     \centering \includegraphics[width=\textwidth,trim={3cm 1cm 3cm 1cm},clip]{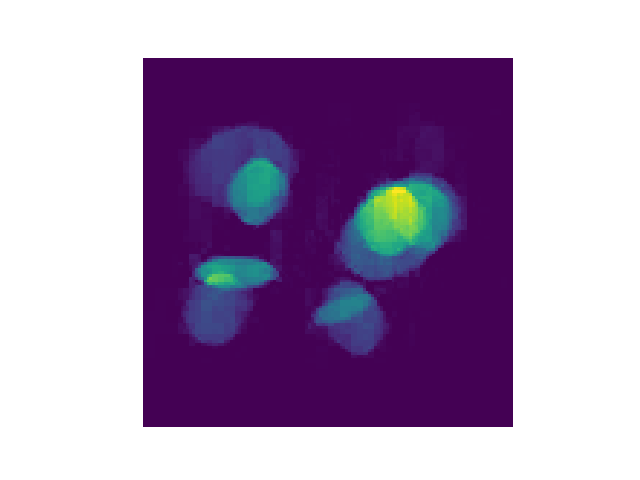}
     \caption{$\mathbf{u}_{\text{$\Psi$do-F}}$ \\ RE: 0.20, SSIM: 0.82}
     \label{fig:im10016ourprediction_flbcm3}
  \end{subfigure}
  \begin{subfigure}[b]{0.3\textwidth}
     \centering \includegraphics[width=\textwidth,trim={3cm 1cm 3cm 1cm},clip]{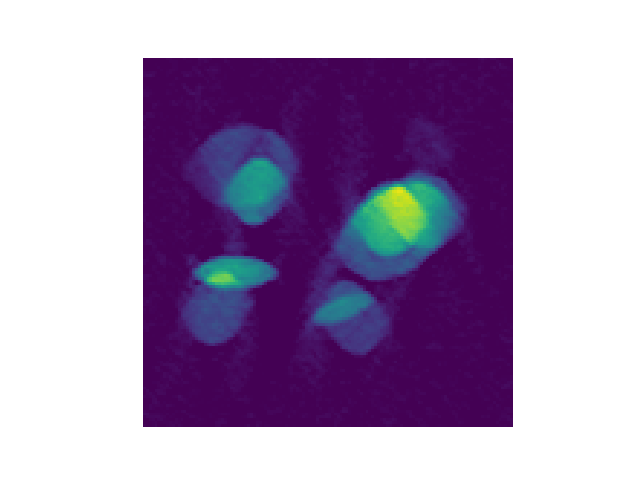}
     \caption{$\mathbf{u}_{\text{$\Psi$do-F}}^+$ \\ RE: 0.24, SSIM: 0.58}
     \label{fig:im10016ourprediction_flbcm3exp}
  \end{subfigure}
  
  \begin{subfigure}[b]{0.3\textwidth}
     \centering \includegraphics[width=\textwidth,trim={3cm 1cm 3cm 1cm},clip]{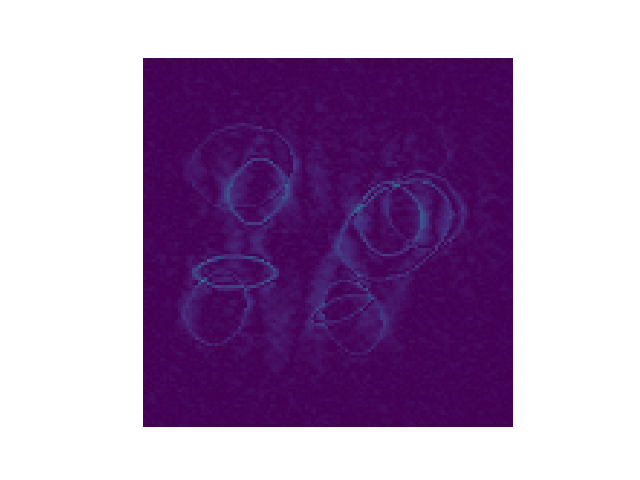}
     \caption{ $|\mathbf{u}^\dagger-\mathbf{u}_{\text{$\Psi$do-O}}^+|$  }
     \label{fig:im10016ourprediction_flacm1exp_err}
  \end{subfigure}
  \begin{subfigure}[b]{0.3\textwidth}
     \centering \includegraphics[width=\textwidth,trim={3cm 1cm 3cm 1cm},clip]{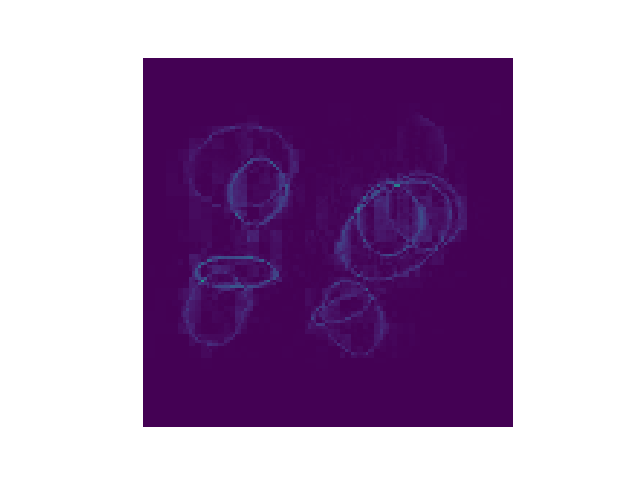}
     \caption{ $|\mathbf{u}^\dagger-\mathbf{u}_{\text{$\Psi$do-F}}|$  }
     \label{fig:im10016ourprediction_flbcm3_err}
  \end{subfigure}
  \begin{subfigure}[b]{0.3\textwidth}
     \centering \includegraphics[width=\textwidth,trim={3cm 1cm 3cm 1cm},clip]{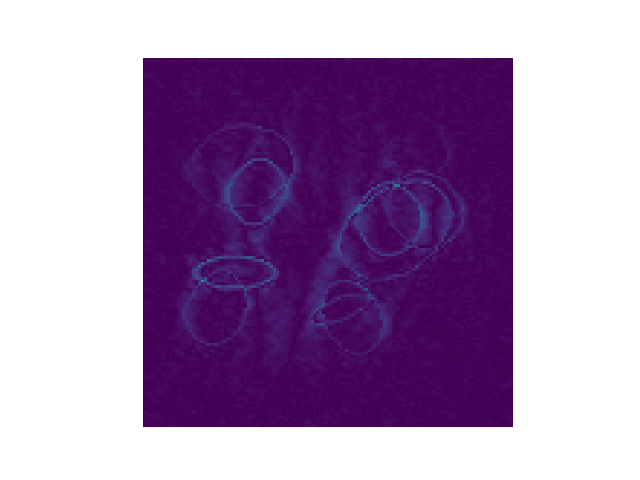}
     \caption{$|\mathbf{u}^\dagger-\mathbf{u}_{\text{$\Psi$do-F}}^+|$ }
     \label{fig:im10016ourprediction_flbcm3exp_err}
  \end{subfigure} 
  \caption{Visualization of the results for one test image.}
  \label{fig:compar_methods_10016}
\end{figure}
\begin{figure}[H]
  \centering
  \captionsetup{justification=centering}
  \begin{subfigure}[b]{0.3\textwidth}
     \centering \includegraphics[width=\textwidth,trim={3cm 1cm 3cm 1cm},clip]{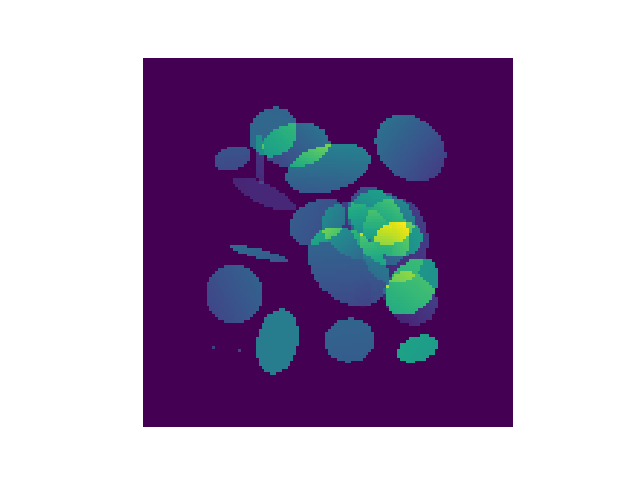}
     \caption{Ground truth $\mathbf{u}^\dagger$  \\{\color{white} RE: , SSIM: }}
     \label{fig:im10004groundtruth}
  \end{subfigure}
  \begin{subfigure}[b]{0.3\textwidth}
     \centering \includegraphics[width=\textwidth,trim={3cm 1cm 3cm 1cm},clip]{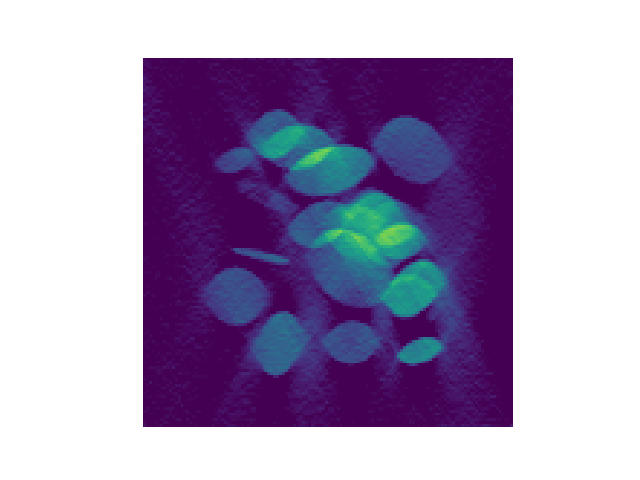}
     \caption{$\mathbf{u}_{\text{ista}}$ \\ RE: 0.43, SSIM: 0.32}
     \label{fig:im10004ista}
  \end{subfigure}
  \begin{subfigure}[b]{0.3\textwidth}
     \centering \includegraphics[width=\textwidth,trim={3cm 1cm 3cm .5cm},clip]{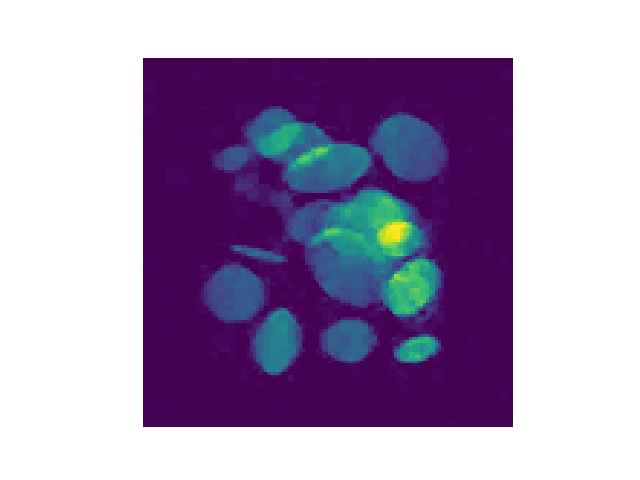}
     \caption{$\mathbf{u}_{\text{$\Psi$do-O}}$ \\ RE: 0.23 , SSIM: 0.78}
     \label{fig:im10004ourprediction_flacm1}
  \end{subfigure}
  
  \begin{subfigure}[b]{0.3\textwidth}
     \centering \includegraphics[width=\textwidth,trim={3cm 1cm 3cm 1cm},clip]{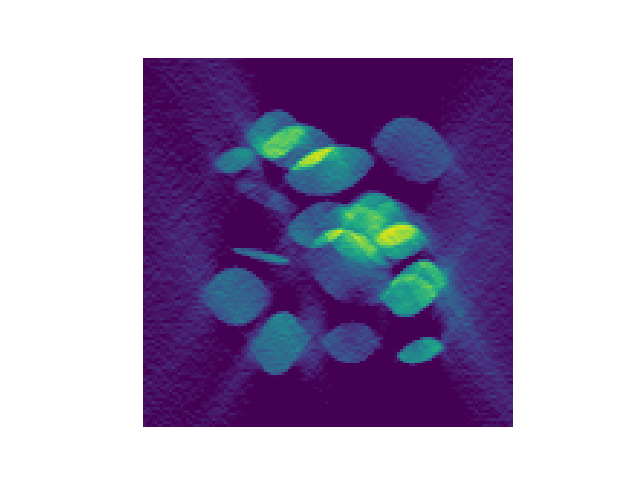}
     \caption{$\mathbf{u}_{\text{FBP}}$  \\ RE: 0.64, SSIM: 0.18}
     \label{fig:im10004fbp}
  \end{subfigure}
  \begin{subfigure}[b]{0.3\textwidth}
     \centering \includegraphics[width=\textwidth,trim={3cm 1cm 3cm 1cm},clip]{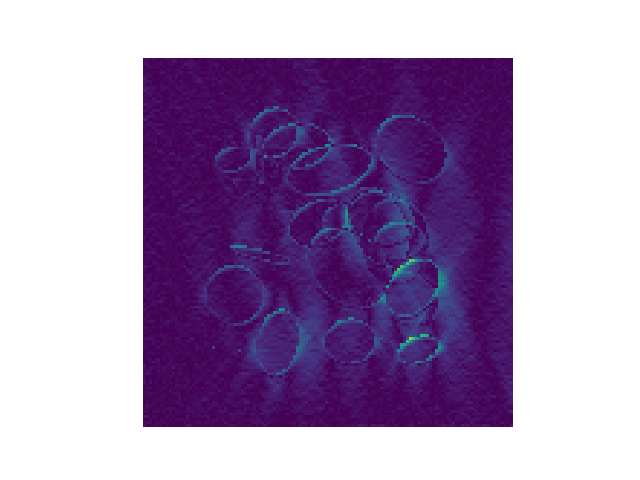}
     \caption{$|\mathbf{u}^\dagger-\mathbf{u}_{\text{ista}}|$ \\{\color{white} RE: , SSIM: } }
     \label{fig:im10004ista_err}
  \end{subfigure}
  \begin{subfigure}[b]{0.3\textwidth}
     \centering \includegraphics[width=\textwidth,trim={3cm 1cm 3cm .5cm},clip]{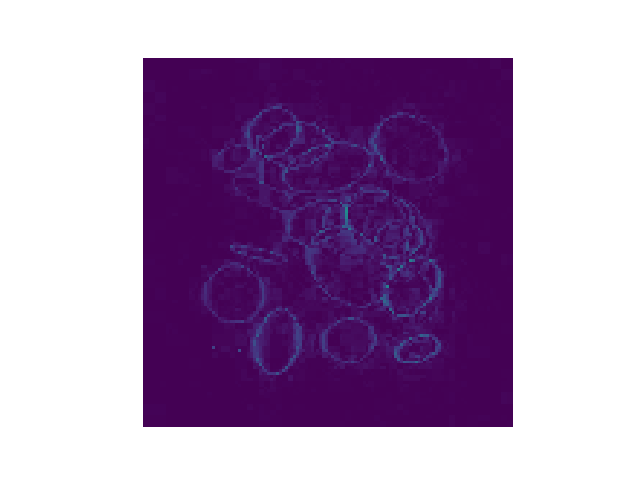}
     \caption{ $|\mathbf{u}^\dagger-\mathbf{u}_{\text{$\Psi$do-O}}|$ \\{\color{white} RE: , SSIM: }}
     \label{fig:im10004ourprediction_flacm1_err}
  \end{subfigure}
  
  \begin{subfigure}[b]{0.3\textwidth}
     \centering \includegraphics[width=\textwidth,trim={3cm 1cm 3cm 1cm},clip]{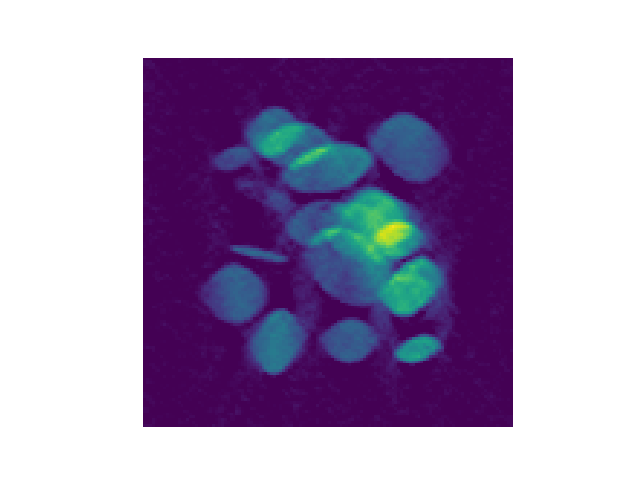}
     \caption{$\mathbf{u}_{\text{$\Psi$do-O}}^+$ \\ RE: 0.28, SSIM: 0.56}
     \label{fig:im10004ourprediction_flacm1exp}
  \end{subfigure}
  \begin{subfigure}[b]{0.3\textwidth}
     \centering \includegraphics[width=\textwidth,trim={3cm 1cm 3cm 1cm},clip]{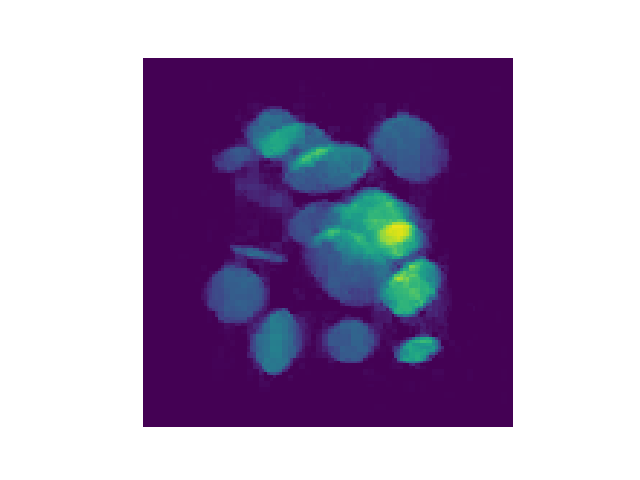}
     \caption{$\mathbf{u}_{\text{$\Psi$do-F}}$ \\ RE: 0.25, SSIM: 0.76}
     \label{fig:im10004ourprediction_flbcm3}
  \end{subfigure}
  \begin{subfigure}[b]{0.3\textwidth}
     \centering \includegraphics[width=\textwidth,trim={3cm 1cm 3cm 1cm},clip]{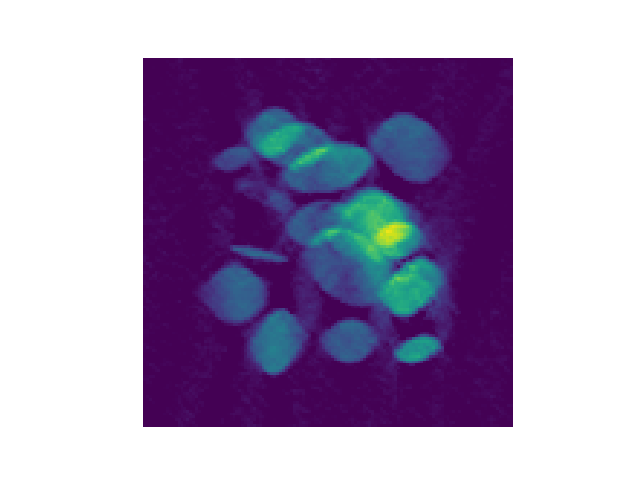}
     \caption{$\mathbf{u}_{\text{$\Psi$do-F}}^+$ \\ RE: 0.29, SSIM: 0.53 }
     \label{fig:im10004ourprediction_flbcm3exp}
  \end{subfigure}
  
  \begin{subfigure}[b]{0.3\textwidth}
     \centering \includegraphics[width=\textwidth,trim={3cm 1cm 3cm 1cm},clip]{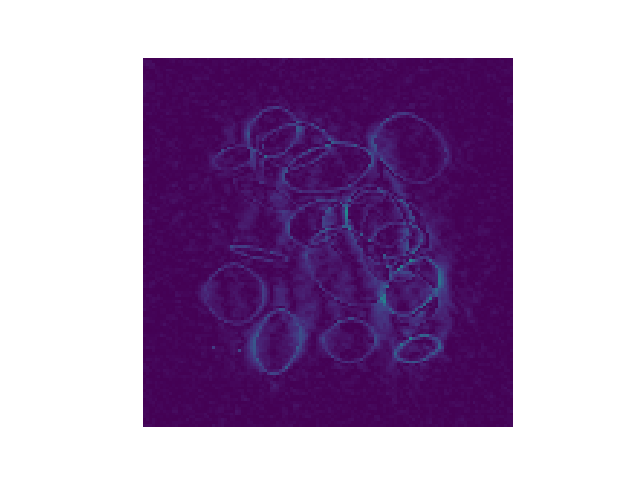}
     \caption{ $|\mathbf{u}^\dagger-\mathbf{u}_{\text{$\Psi$do-O}}^+|$  }
     \label{fig:im10004ourprediction_flacm1exp_err}
  \end{subfigure}
  \begin{subfigure}[b]{0.3\textwidth}
     \centering \includegraphics[width=\textwidth,trim={3cm 1cm 3cm 1cm},clip]{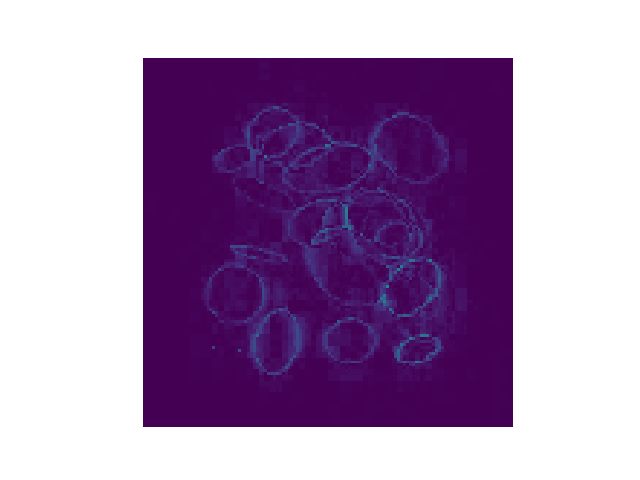}
     \caption{ $|\mathbf{u}^\dagger-\mathbf{u}_{\text{$\Psi$do-F}}|$  }
     \label{fig:im10004ourprediction_flbcm3_err}
  \end{subfigure}
  \begin{subfigure}[b]{0.3\textwidth}
     \centering \includegraphics[width=\textwidth,trim={3cm 1cm 3cm 1cm},clip]{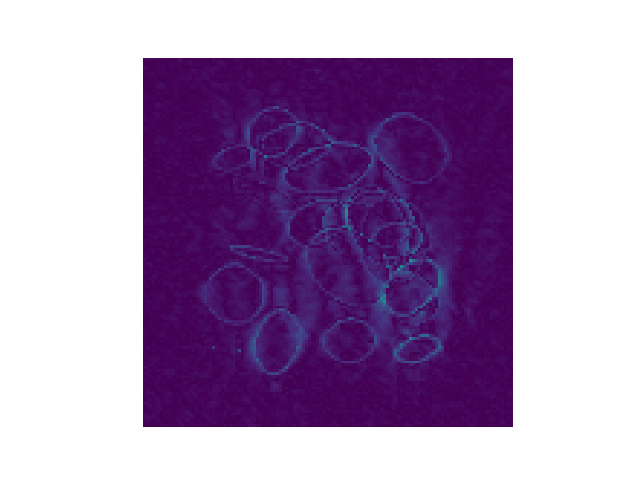}
     \caption{$|\mathbf{u}^\dagger-\mathbf{u}_{\text{$\Psi$do-F}}^+|$ }
     \label{fig:im10004ourprediction_flbcm3exp_err}
  \end{subfigure} 
  \caption{Visualization of the results for one test image.}
  \label{fig:compar_methods_10004}
\end{figure}
\begin{table}[!t]
\centering
 \begin{tabular}{c |c |c |c|c} 
 \hline\hline
 Method & RE & PSNR & SSIM & HaarPSI\\ [1ex] 
 \hline\hline
 $\mathbf{u}_{\text{ista}}$ & 0.44 & 22.84 & 0.36 & 0.37 \\ 
 $\mathbf{u}_{\text{FBP}}$ & 0.64 & 19.49  & 0.20 & 0.30  \\
 \hdashline
 $\mathbf{u}_{\text{$\Psi$do-F}}^+$ & 0.29 & 26.63 & 0.59 & 0.47  \\
 $\mathbf{u}_{\text{$\Psi$do-F}}$ & 0.25 & 27.63 & 0.78 & 0.54  \\
 $\mathbf{u}_{\text{$\Psi$do-O}}^+$ &  0.28 & 26.76 & 0.60 & 0.48  \\ 
 $\mathbf{u}_{\text{$\Psi$do-O}}$ & \textbf{0.23} & \textbf{28.43} & \textbf{0.81} &\textbf{0.58 } \\ 
 \hline\hline
\end{tabular}
\caption{Comparison of reconstruction methods. The similarity values are averaged over
the 500 images of the test set.}
\label{tab:average_im_quality}
\end{table}
With our two models, $\Psi$DONet-F and $\Psi$DONet-O, whether with positivity constraint on the soft-thresholding parameter or without, it is possible to substantially reduce those artifacts and contrast issues. As it can be seen in \cref{fig:compar_methods_10016} and \cref{fig:compar_methods_10004}, our proposed methods lead to undeniably enhanced reconstructions, with a meaningful diminution of the relative error. In particular, $\Psi$DONet-O provides slightly better similarity values than $\Psi$DONet-F, although both implementations produces comparable results.

In the case where the positivity of the soft-thresholding parameter is enforced, that is for $\mathbf{u}_{\text{$\Psi$do-F}}^+$ and  $\mathbf{u}_{\text{$\Psi$do-O}}^+$, one can notice that the streaking artifacts, although greatly lessened when compared with the ISTA images, are still present on the reconstructions (cf \cref{fig:im10016ourprediction_flacm1exp}, \cref{fig:im10016ourprediction_flbcm3exp}, \cref{fig:im10004ourprediction_flacm1exp} and
\cref{fig:im10004ourprediction_flbcm3exp}). In fact, the SSIM measures are greater than in the ISTA case, but still clearly below the SSIM values obtained with the non-constrained version of the two models. 
The latter ($\mathbf{u}_{\text{bowtie}}$ and  $\mathbf{u}_{\text{$\Psi$do}}$)  do a noteworthy job in removing the artifacts and sharpening the edges (cf \cref{fig:im10016ourprediction_flacm1}, \cref{fig:im10016ourprediction_flbcm3}, \cref{fig:im10004ourprediction_flacm1} and
\cref{fig:im10004ourprediction_flbcm3}). 

Overall, $\Psi$DONet-O without any constraint on the soft-thresholding parameters offers the best results among the compared methods and allows for an optimized implementation of $\Psi$DONet.

\section{Conclusions}
\label{sec:conclusions}
In the present paper, we introduced a novel CNN, named $\Psi$DONet, inspired by the well-known ISTA and the convolutional nature of certain FIOs and $\Psi$DOs, like the limited angle Radon transform. We proved that the unrolled iterations of ISTA can be interpreted as layers of a CNN, where the downsampling, upsampling and convolution operations, typically defining a CNN, can be exactly specified by combining the convolutional nature of the limited angle Radon transform and basic properties defining an orthogonal wavelet system. In addition, we proved that, for a specific choice of the parameters involved, $\Psi$DONet recovers standard ISTA or a perturbation of ISTA, up to a bound on the filters coefficients which we estimated in the case of limited angle Radon transform. 

The key feature of the proposed architecture is its potential to learn $\Psi$DO-like structures, which makes it suitable to be extended to any convolutional operator which is a FIO  or  $\Psi$DO. Moreover, the analysis carried out in paper allows to gain understanding and interpretability of the results, which gives insight into a whole class of inverse problems arising from FIO or $\Psi$DO and opens up for fundamental theoretical generalization results.  

As a proof of concept, we tested two different implementations of $\Psi$DONet on simulated data from limited angle geometry, generated from the ellipse data set. The improvement, compared to standard ISTA (and classical FBP) is notable and it is promising for further numerical testing which we leave to future work. Additional directions for future numerical testing include larger sizes for images, smaller and sparser visible wedges, additional regularization for the loss function and reconstructions from real data.


\appendix
\section{Proof of \cref{prop:conv2}} 
\label{Appendix1}
\begin{proof}
 According to \cite[Section 3]{flemming2018injectivity}, 
 the following variational source condition is satisfied by every $w \in \ell^1(\N)$:
\begin{equation}
	\beta \| w - w^\dag \|_{\ell^1} \leq \| w \|_{\ell^1} - \| w^\dag \|_{\ell^1} + C \|A W^* w- A W^* w^\dag \|_{Y}.
\label{eq:vsc}
\end{equation}
We aim at applying it to $w = w_{\delta,\vN} \in W_p \subset \ell^1(\N)$. \par
First consider the term $\| w \|_{\ell^1} - \| w^\dag \|_{\ell^1}$ in the right-hand side. 
Since $w_{\delta,\vN}$ is a solution of \eqref{eq:minN},
\[
\begin{aligned}
    \lambda \|w_{\delta,\vN} \|_{\ell^1} =  \left( \| A_\vN W^* w_{\delta,\vN} - \mathbb{P}_q m\|_{Y}^2+\lambda \|w_{\delta,\vN}\|_{\ell^1} \right)  - \| A_\vN W^* w_{\delta,\vN}- \mathbb{P}_q m\|_{Y}^2 \\
    \leq \| A_\vN W^* \mathbb{P}_p w^\dag - \mathbb{P}_q m\|_{Y}^2+\lambda \|\mathbb{P}_p w^\dag \|_{\ell^1}  - \| A_\vN W^* w_{\delta,\vN} - \mathbb{P}_q m\|_{Y}^2,
\end{aligned}
\]
whence
\[
    \|w_{\delta,\vN} \|_{\ell^1} - \|w^\dag \|_{\ell^1} \leq  \frac{1}{\lambda} \| A_\vN W^* \mathbb{P}_p w^\dag - \mathbb{P}_q m \|_{Y}^2 - \frac{1}{\lambda}  \| A_\vN W^* w_{\delta,\vN}- \mathbb{P}_q m\|_{Y}^2.
\]
We can easily check that $A_\vN W^* \mathbb{P}_p = \mathbb{P}_q A W^* \mathbb{P}_p$; then, since $\| \mathbb{P}_q \|_{Y\rightarrow Y}\leq 1$, denoting by $Q = \| A_\vN W^* w_{\delta,\vN}- \mathbb{P}_q m\|_{Y}$, we have
\[
\begin{aligned}
    \|w_{\delta,\vN} \|_{\ell^1} - \|w^\dag \|_{\ell^1} &\leq  \frac{1}{\lambda} \| A W^* \mathbb{P}_p w^\dag - m \|_{Y}^2 - \frac{1}{\lambda}  Q ^2 \\
    &\leq \frac{1}{\lambda} \| A W^* (\mathbb{P}_p w^\dag - w^\dag)\|_{Y}^2 + \frac{1}{\lambda} \| A W^* w^\dag - m\|_{Y}^2 - \frac{1}{\lambda}  Q^2.
\end{aligned}
\]
In conclusion,
\begin{equation}
\|w_{\delta,\vN} \|_{\ell^1} - \|w^\dag \|_{\ell^1} \leq \frac{1}{\lambda} \| A \|^2 \|w^\dag - \mathbb{P}_p w^\dag \|^2_{\ell^2} + \frac{1}{\lambda} \delta^2 - \frac{1}{\lambda} Q^2.
    \label{eq:aux1}
\end{equation}
The second term in the right-hand side of \eqref{eq:vsc}, instead, can be bounded as follows:
\begin{equation}
\begin{aligned}
    &\|A W^* (w_{\delta,\vN} - w^\dag) \|_{Y}  \\
    & \qquad = \|\mathcal{P}_q A W^* (w_{\delta,\vN} - w^\dag) \|_{Y} + \|(I-\mathcal{P}_q) A W^* (w_{\delta,\vN} - w^\dag) \|_{Y} \\
    & \qquad \leq \|A_\vN W^* w_{\delta,\vN} - \mathbb{P}_q m \|_{Y} + \delta  + \| (I-\mathbb{P}_q) A \|_{X \rightarrow Y} \| w_{\delta,\vN} - w^\dag  \|_{\ell^1} + \delta \\
    & \qquad \leq Q + M \| (I-\mathbb{P}_q)A \|_{X \rightarrow Y} +  \delta,
\end{aligned}
    \label{eq:aux2}
\end{equation}
where the positive constant $M$ depends on $\| w^\dag\|_{\ell^2}$. In order to get an estimate for $Q = \| A_\vN W^* w_{\delta,\vN}- \mathbb{P}_q m\|_{Y}$, we use \eqref{eq:vsc}: since $0 \leq \beta \| w_{\delta,\vN} - w^\dag \|_{\ell^1}$, using \eqref{eq:aux1} and \eqref{eq:aux2} we have
\[
0 \leq \frac{1}{\lambda} \| A \| \| w^\dag - \mathbb{P}_p w^\dag\|_{\ell^2}^2 + \frac{1}{\lambda}\delta^2 - \frac{1}{\lambda} Q^2 + Q + M \| (I- \mathbb{P}_q)A \|_{X\rightarrow Y} + \delta.
\]
By solving this second-order inequality we get
\begin{equation}
\begin{aligned}
    Q & \leq \frac{\lambda}{2} + \frac{\lambda}{2}\left( 1+ \frac{4}{\lambda^2} \| A\|^2 \| w^\dag - \mathbb{P}_p w^\dag\|_{\ell^2}^2 \frac{4}{\lambda}\delta^2 + \frac{4M}{\lambda} \| (I- \mathbb{P}_q)A\|_{X\rightarrow Y} \frac{4}{\lambda}\delta \right)^{\frac{1}{2}} \\
    &\leq \lambda + \delta + \|A\| \| w^\dag - \mathbb{P}_p w^\dag\|_{\ell^2} + M \| (I- \mathbb{P}_q)A \|_{X\rightarrow Y} 
\end{aligned}
    \label{eq:q}
\end{equation}
Combining \eqref{eq:vsc}, \eqref{eq:aux1}, \eqref{eq:aux2}, and \eqref{eq:q} we easily conclude the proof.
\end{proof}

\section{Proof of \cref{prop:conv3}}
\label{Appendix2}
\begin{proof}
Consider the sequence $e_n = \| \wt{n+1} - w^{(n+1)} \|_{\ell^2}$. Thanks to the nonexpansivity of the operator $S_{\lambda/L}$, it holds that 
\begin{equation} \label{eq:e0}
\begin{aligned}
e_0 &= \| \wt{1} - w^{(1)} \|_{\ell^2} = \| \mathcal{T_Z}(w^{(0)}) - \mathcal{T}(w^{(0)}) \|_{\ell^2} \\
& \leq \frac{1}{L} \| W A_\vN^* A_\vN W^* - Z \| \| w^{(0)} \|_{\ell^2} \leq \frac{1}{L} \rho \| w^{(0)} \|_{\ell^2}.
\end{aligned}
\end{equation}
Analogously, for $n \geq 1$,
\begin{equation} \label{eq:el}
\begin{aligned}
e_n & = \| \mathcal{T}(w^{(n)}) - \mathcal{T_Z}(\wt{n}) \|_{\ell^2} \\
& \leq  \| I - \frac{1}{L} Z \| \| w^{(n)} - \wt{n} \|_{\ell^2} + \frac{1}{L} \| W A_\vN^* A_\vN W^* - Z \| \| w^{(n)} \|_{\ell^2} \\
& \leq \| I - \frac{1}{L} Z \| e_{n-1} + \frac{1}{L} \rho \| w^{(n)} \|_{\ell^2}.
\end{aligned}
\end{equation}
Since $L \geq \| W A_\vN^* A_\vN W^*  \|$, then
\[
 \| I - \frac{1}{L} Z \| \leq \| I - \frac{1}{L} W A_\vN^* A_\vN W^*  \| + \frac{1}{L} \| W A_\vN^* A_\vN W^*  - Z\| \leq 1 + \frac{1}{L} \rho.
\]
Moreover, since the sequence $\{w^{(n)}\}$ is convergent, then it is also bounded: let, \textit{e.g.}, $\| w^{(n)}\|_{\ell^2} \leq M$. As a consequence of \eqref{eq:e0},\eqref{eq:el},
\[
e_N \leq \sum_{n=0}^N \left( 1+\frac{1}{L}\rho \right)^{N-n} \frac{1}{L}\rho \| w^{(n)}\|_{\ell^2} 
\leq  M \left( \big(1+\frac{1}{L}\rho \big)^{N+1}-1 \right) 
\]
Let now $N \geq N_0$ and $\rho N \leq \eta_0$: then, with a constant $c = c(N_0,\eta_0)$, it holds:
\[
\| \wt{N} - w^{(N)} \|_{\ell^2} = e_N \leq M(e^{\frac{1}{L} \rho N} - 1) \leq  c \frac{M}{L} \rho N. 
\]
Combining this result with \eqref{eq:BL}, we can guarantee that
\[
\| \wt{N} - w_{\delta,\vN} \|_{\ell^2} \leq \| \wt{N} - w^{(N)} \|_{\ell^2} + \| w^{(N)} - w_{\delta,\vN} \|_{\ell^2} \leq c_3 a^N + \tilde{c}_4 \rho N,
\]
which proves \eqref{eq:conv4Old}. To obtain \eqref{eq:conv4}, simply substitute $N = \log_a\delta$ and $\rho = \frac{\delta}{N}$ and consider $c_4 = c_3 + \tilde{c}_4$.
\end{proof}

\bibliographystyle{siam}
\bibliography{references}
\end{document}